\documentclass{elsarticle}
\usepackage{graphicx} 
\usepackage{natbib}
\usepackage{graphicx,amsmath,amsthm,amsfonts,color,amssymb,comment}
\usepackage{tikz}
\usetikzlibrary {arrows.meta}
\usepackage{subfiles} 
\usepackage[nameinlink]{cleveref} 
\usepackage{kpfonts}
\usepackage{etoolbox}
\usepackage{thmtools}
\usepackage{environ}
\usepackage{float}

\def\bfm#1{\boldsymbol{#1}} 

\newtheorem{theorem}{Theorem}[section]

\theoremstyle{definition}

\begin{document}

\begin{frontmatter}

\title{\textbf{A Quadratic \(G^1\) Spline Approximation of the Sphere over Uniform Polyhedra}}
\cortext[corauth]{Corresponding author}
\author[address1]{Ema \v{C}e\v{s}ek\corref{corauth}}
\ead{ema.cesek@fmf.uni-lj.si}

\author[address1,address2]{Ale\v{s} Vavpeti\v{c}}
\ead{ales.vavpetic@fmf.uni-lj.si}

\address[address1]{University of Ljubljana, Faculty of Mathematics and Physics,
Jadranska 19, Ljubljana, Slovenia}
\address[address2]{Institute of Mathematics, Physics and Mechanics, Jadranska 19, Ljubljana, Slovenia}
\date{}

\begin{abstract}
In this paper, we study geometrically continuous quadratic splines over triangulations. While a rich variety of $C^1$ quadratic splines is available over planar domains, and such splines can also be constructed on the torus, the problem becomes significantly more challenging on more general surfaces.

We first construct a $G^1$ spline over a regular spherical $n$-gon, subdivided into $3n$ triangles. Based on this construction, we obtain a quadratic $G^1$ spline approximation of the sphere induced by an arbitrary uniform polyhedron, where each $n$-gonal face is subdivided into $3n$ triangles. The construction uses only quadratic triangular patches and yields explicit control points depending on the geometry of the underlying polyhedron and one free parameter. We also analyze the resulting approximation quality and curvature behavior, and illustrate the construction on Platonic and Archimedean solids.
\end{abstract}

\begin{keyword}
geometric interpolation\sep geometric continuity\sep triangular parametric polynomial patch\sep sphere approximation\sep uniform solids    
\MSC[2020] 65D05 \sep  65D07 \sep 65D17
\end{keyword}

\end{frontmatter}

\section{Introduction}
The sphere is one of the fundamental geometric objects in computer-aided geometric design, geometric modeling, and approximation theory. It is well known that the sphere admits rational parametric representations~\cite{Schicho-98}, whereas no exact polynomial parametric representation exists. Since rational parameterizations can be difficult to handle in practical applications, polynomial approximations of the sphere remain an important topic in geometric modeling. It is sufficient to consider the unit sphere, since any other sphere can be obtained from it by translation and scaling, without affecting the approximation quality.

A natural strategy for approximating the sphere is to construct parametric polynomial spline surfaces based on polynomial approximations of spherical triangles forming a triangulation of the sphere or of a portion thereof. Such spline constructions should exhibit not only continuity but also an appropriate degree of smoothness. In geometric modeling, geometric continuity is often sufficient, since the resulting surfaces are primarily intended for visualization and shape design. Consequently, \(G^k\)-continuous splines~\cite{FarinHoschekKim-02-Handbook,Kiciak-2017} are particularly attractive. Compared with \(C^k\)-continuous splines, they offer additional degrees of freedom that can be exploited either to improve approximation quality or to enhance shape control.

The construction of low-degree geometrically continuous spline surfaces has a long history and remains an active area of research. Classical works on geometric continuity and smooth patch constructions include~\cite{Gregory-1989,Farin-2002,Peters-Reif-2008}. Multi-patch surface geometries represented by $G^1$ smooth parameterizations are also of central importance in isogeometric analysis, where $C^1$-continuous spline functions defined over such domains are required for the discretization of higher-order partial differential equations. Low-degree $G^1$ representations of surface domains are particularly attractive, as they give rise to low-degree gluing functions, which provide more degrees of freedom and lead to higher-dimensional isogeometric spline spaces with better approximation properties. Recent work~\cite{GroseljKaplKnezTakacsVitrih2024} has shown that even quadratic parameterizations of mesh elements admit nontrivial $C^1$ smooth spline constructions, in contrast to linear parameterizations. For $G^1$-smooth surface domains, however, quadratic parameterizations constitute the lowest-degree case for which such constructions are possible. Although rich classes of smooth spline constructions over planar triangulations are known~\cite{Lai-Schumaker-2007}, the situation for more general surfaces is considerably more restrictive. In particular, low-degree polynomial patches impose severe geometric constraints on $G^1$ continuity. This phenomenon was analyzed in detail in~\cite{Karciauskas-Peters-2025}, where the authors studied which smooth surfaces can be constructed from splines of total degree~2.

Early work on the geometric interpolation of general surfaces can be found in~\cite{Morken-2005-parametric-surface}, while specific interpolation schemes were studied in~\cite{Jaklic-Kozak-Krajnc-Vitrih-Zagar-2006-patch,Jaklic-Kozak-Krajnc-Vitrih-Zagar-Geometric-08}. More recently, a novel approach to geometric interpolation by parametric polynomials was introduced in~\cite{Jaklic-Kanduc-2017}. A \(G^1\) approximation of the unit sphere by biquadratic polynomial spline patches over a cube was constructed in~\cite{VAVPETIC2023128196}. In contrast, it was shown in~\cite{VAVPETIC2022102061} that no \(G^1\) approximation of the unit sphere exists using quadratic polynomial spline patches over triangulations induced by a tetrahedron, octahedron, or icosahedron inscribed in the unit sphere.

Approximations based on cubic and quartic triangular patches have also been proposed. In particular, it was shown in~\cite{Ahn-26} that the best cubic approximants of equilateral spherical triangles are LN surfaces. Since every quadratic triangular patch is also an LN surface~\cite{Peternell-Odehnal-2008}, quadratic patches constitute a natural low-degree framework for approximating spherical regions.

The use of Platonic solids as underlying polyhedra for smooth spline constructions dates back to~\cite{Peters-97}, where families of symmetric polynomial spheroids based on Platonic solids and their duals were introduced. The resulting constructions employ either cubic triangular patches or biquadratic quadrilateral patches.

In this paper, we further reduce the polynomial degree by constructing $G^1$ spline approximations of the sphere entirely from quadratic triangular patches. Besides their relevance in geometric modelling, such constructions may also serve as geometric domains supporting low-degree smooth spline spaces in isogeometric analysis.

The paper is organized as follows. In~\Cref{sec:Preliminaries}, we present a theorem characterizing symmetric \(G^1\) splines consisting of two triangular patches. This theorem is then used in~\Cref{sec:SphericalCap}, where a \(G^1\) spline over a regular spherical polygon is constructed. We also show that such splines over spherical polygons cannot be used directly to construct a \(G^1\) spline over the entire sphere. In~\Cref{sec:Triangulation}, we construct a \(G^1\) spline over a suitable triangulation of a regular spherical polygon. This construction can be extended to a \(G^1\) spline over the whole sphere whose underlying mesh is obtained from an arbitrary uniform polyhedron. In~\Cref{sec:Generalisation}, we discuss possible generalisations. The paper concludes with~\Cref{sec:Conclusion}.

\section{Preliminaries}\label{sec:Preliminaries}
The main goal of this paper is to construct a \(G^1\)-continuous approximation of the sphere using parametric polynomial splines composed of quadratic triangular patches. Without loss of generality, it suffices to consider the unit sphere \(\mathcal{S}\) centered at the origin, since any other sphere can be obtained from \(\mathcal{S}\) by translation and scaling. Let \(\mathcal{P}\) be a convex uniform polyhedron inscribed in \(\mathcal{S}\). Its faces are regular \(n\)-gons, whose radial projections onto \(\mathcal{S}\) form a tiling of spherical regular \(n\)-gons. To approximate \(\mathcal{S}\), it is therefore sufficient to construct an approximation of each spherical regular \(n\)-gon.

A key advantage of our approach is its locality: the approximation over a given face of \(\mathcal{P}\) can be constructed independently of the other faces. In the next section, we introduce a triangulation of the faces of \(\mathcal{P}\) into isosceles triangles. The problem is thereby reduced to approximating spherical isosceles triangles obtained by radially projecting these triangles onto \(\mathcal{S}\). The approximation is represented by quadratic triangular Bézier patches defined over the resulting subtriangles. To ensure a visually smooth surface, \(G^1\) continuity conditions are imposed across the boundaries of adjacent patches. 

Let \(\Delta\) be the standard 2-simplex defined by \(\Delta = \{ (\xi,\eta)\in \mathbb{R}^2 \mid \xi,\eta\geq 0,\ \xi+\eta \leq 1 \}\). Let \(\bfm{p},\bfm{q} \colon \Delta \to \mathbb{R}^3\) be quadratic triangular Bézier patches parameterized by
\begin{align*}
\bfm{p}(\xi,\eta) =  \sum_{i+j+k=2} B_{i,j,k}(\xi,\eta) \bfm{P}_{i,j,k},
\quad\text{ and }\quad
\bfm{q}(\xi,\eta) =  \sum_{i+j+k=2} B_{i,j,k}(\xi,\eta) \bfm{Q}_{i,j,k}, \quad i, j, k \in \mathbb{Z}_{+}, 
\end{align*}
where
\[
B_{i,j,k}(\xi,\eta) = \frac{2}{i!\,j!\,k!}\,\xi^i \eta^j (1-\xi-\eta)^k,
\]
are the Bernstein polynomials of degree \(2\), and \(\bfm{P}_{i,j,k}, \bfm{Q}_{i,j,k} \in \mathbb{R}^3\) are the corresponding control points. We assume that the patches \(\bfm{p}\) and \(\bfm{q}\) share a common boundary curve along \(\eta=0\), i.e., \(\bfm{p}(\xi,0)=\bfm{q}(\xi,0)\). This implies that \(\bfm{P}_{i,0,2-i}=\bfm{Q}_{i,0,2-i}\), \(i=0,1,2\).

Two patches with a common boundary curve are $G^1$ continuous if their tangent plane varies continuously along that common boundary curve. Equivalently, the vectors $\frac{\partial{\bfm{p}}}{\partial{\xi}}(\xi,0)=\frac{\partial{\bfm{q}}}{\partial{\xi}}(\xi,0)$, $\frac{\partial{\bfm{p}}}{\partial{\eta}}(\xi,0)$, and $\frac{\partial{\bfm{q}}}{\partial{\eta}}(\xi,0)$ must be coplanar for every \(\xi\in[0,1]\) (see~\cite[(3.3) on page 41]{Kiciak-2017}). If the two patches are mirror images of each other with respect to the plane containing their common boundary curve, then the \(G^1\) continuity conditions admit a particularly simple form, as described in the theorem below. The equivalence of the first and third conditions is stated as an exercise in~\cite[page 53]{Kiciak-2017}. 

Before stating the theorem, we introduce the following notation. Let us define $u_i =\bfm{P}_{i+1,0,1-i} - \bfm{P}_{i,0,2-i}$, where $i=0,1$, $v_{jk}=\bfm{P}_{j,k+1,1-j-k} - \bfm{P}_{j,k,2-j-k}$ and $w_{jk} = \bfm{Q}_{j,k+1,1-j-k} - \bfm{Q}_{j,k,2-j-k}$, where $j+k = 0,1$ (see~\Cref{fig:vectors}). 
\begin{figure}[h!]
\centering
\begin{tikzpicture}[scale=0.8]
\coordinate (p00) at (0,0,0);
\coordinate (p10) at (0,3.6,1.2);
\coordinate (p20) at (0,5,-1);

\coordinate (q01) at (3,0.5,0);
\coordinate (p01) at (-3,0.5,0);

\coordinate (q11) at (2.5,5,0);
\coordinate (p11) at (-2.5,5,0);

\coordinate (q02) at (5,2.5,0);
\coordinate (p02) at (-5,2.5,0);

\draw[line width=0.01mm,-{LaTeX[width'=0pt .5, length=10pt]}] (p00)-- node[left]{$u_0$} (p10);
\draw[line width=0.01mm,-{LaTeX[width'=0pt .5, length=10pt]}] (p10) --node[left]{$u_1$} (p20);

\draw[line width=0.01mm,-{LaTeX[width'=0pt .5, length=10pt]}] (p00)-- node[above]{$v_{00}$} (p01);
\draw[line width=0.01mm,-{LaTeX[width'=0pt .5, length=10pt]}]  (p10)-- node[below, xshift = -6pt]{$v_{10}$} (p11);
\draw[line width=0.01mm,-{LaTeX[width'=0pt .5, length=10pt]}] (p00)-- node[above]{$w_{00}$} (q01);
\draw[line width=0.01mm,-{LaTeX[width'=0pt .5, length=10pt]}]  (p10)-- node[below, xshift = 6pt]{$w_{10}$} (q11);
\draw[line width=0.01mm,-{LaTeX[width'=0pt .5, length=10pt]}] (p01)-- node[above, xshift= 4pt]{$v_{01}$} (p02);
\draw[line width=0.01mm,-{LaTeX[width'=0pt .5, length=10pt]}] (q01)-- node[above, xshift= -5pt]{$w_{01}$} (q02);

\draw (p00) .. controls (p10) .. (p20);
\draw (p00) .. controls (p01) .. (p02);
\draw (p00) .. controls (q01) .. (q02);
\draw (p20) .. controls (p11) .. (p02);
\draw (p20) .. controls (q11) .. (q02);

\draw[line width=0.01mm] (p20)--(p11)--(p02); 
\draw[line width=0.01mm] (p20)--(q11)--(q02); 

\fill (q11) circle (0.05) node [above] {$\bfm{Q}_{1,1,0}$};
\fill (q01) circle (0.05) node [below, yshift=-2pt] {$\bfm{Q}_{0,1,1}$};
\fill (q02) circle (0.05) node [right] {$\bfm{Q}_{0,2,0}$};
\fill (p11) circle (0.05) node [above] {$\bfm{P}_{1,1,0}$};
\fill (p01) circle (0.05) node [below] {$\bfm{P}_{0,1,1}$};
\fill (p02) circle (0.05) node [left] {$\bfm{P}_{0,2,0}$};
\fill (p00) circle (0.05) node [below] {$\bfm{P}_{0,0,2}= \bfm{Q}_{0,0,2}$};
\fill (p10) circle (0.05) node [anchor=north west, yshift= 5pt] {$\bfm{P}_{1,0,1}= \bfm{Q}_{1,0,1}$};
\fill (p20) circle (0.05) node [anchor=north, yshift = 15pt ] {$\bfm{P}_{2,0,0}= \bfm{Q}_{2,0,0}$};
\end{tikzpicture}
\caption{Patches $\bfm{p}$ and $\bfm{q}$ with control points $\bfm{P}_{i,j,2-i-j}$ and $\bfm{Q}_{i,j,2-i-j}$, where $i,j=0,1,2$, and the corresponding vectors between the control points.} 
\label{fig:vectors}
\end{figure}
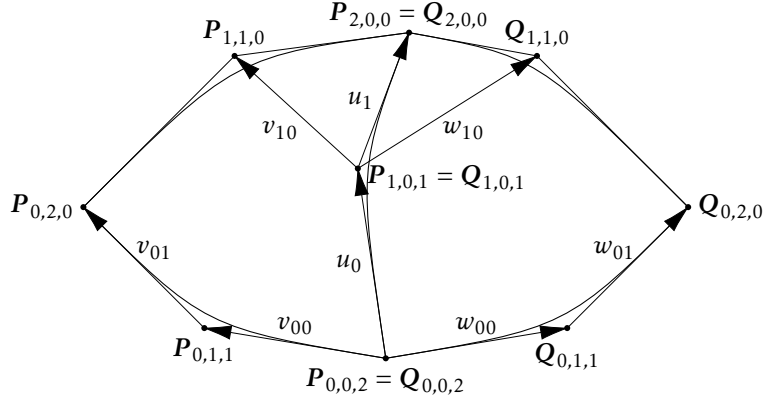
We assume that the control points of the common boundary curve are not collinear and therefore determine a plane $\Pi\subset \mathbb{R}^3$. Let \(\mathcal{R}_{\Pi}\) denote the reflection in the plane \(\Pi\), and let \(\mathcal{P}_{\Pi}\) denote the orthogonal projection onto \(\Pi\). We consider the case where $\bfm{Q}_{i,j,2-i-j} = \mathcal{R}_{\Pi}(\bfm{P}_{i,j,2-i-j})$ for $i+j = 0,1,2$. To avoid unwanted intersections between the patches, we assume that the control points $\bfm{P}_{i,j,2-i-j}$ lie on the same closed half-space determined by $\Pi$, and that only the three control points $\bfm{P}_{i,0,2-i}$, $i=0,1,2$, lie in \(\Pi\). By construction, the plane $\Pi$ is spanned by the vectors $u_0$ and $u_1$, hence the vectors $u_0$, $u_1$, $u_0 \times u_1$ form a basis of $\mathbb{R}^3$.

\begin{theorem}\label{thm:characterisation}
Suppose that the points \(\bfm{P}_{i,0,2-i}\), \(i=0,1,2\), are non-collinear and lie in a plane \(\Pi\). Assume further that the remaining control points \(\bfm{P}_{i,j,2-i-j}\), \(i+j=1,2\) and \(j>0\), lie in the same open half-space determined by \(\Pi\). Let \(\bfm{p}\) be the quadratic Bézier patch defined by the control points \(\bfm{P}_{i,j,2-i-j}\), and let \(\bfm{q}\) be the quadratic Bézier patch defined by the control points \(\bfm{Q}_{i,j,2-i-j}=\mathcal{R}_{\Pi}\bigl(\bfm{P}_{i,j,2-i-j}\bigr)\), obtained by reflecting \(\bfm{P}_{i,j,2-i-j}\) across the plane \(\Pi\). Let \(u_i\), \(v_{jk}\), and \(w_{jk}\) be the vectors defined above. Then the following conditions are equivalent:
\begin{enumerate}
\item The spline formed by \(\bfm{p}\) and \(\bfm{q}\) is \(G^1\) continuous.
\item There exists \(\alpha \in \mathbb{R}\) such that $\mathcal{P}_{\Pi}(v_{00})=\alpha u_0$ and $\mathcal{P}_{\Pi}(v_{10})=\alpha u_1$.
\item The spline formed by \(\bfm{p}\) and \(\bfm{q}\) is \(G^2\) continuous.
\item For every \(\xi\in[0,1]\), the normal vector to the tangent plane at \(\bfm{p}(\xi,0)\) is $\tfrac{\partial}{\partial\xi}\bfm{p}(\xi,0)\times (u_0\times u_1)$.
\end{enumerate}
\end{theorem}

\begin{proof}
Let us write $v_{00}=\alpha_0 u_0+\alpha_1 u_1+\gamma_0 u_0\times u_1$, $w_{00}=\alpha_0 u_0+\alpha_1 u_1-\gamma_0 u_0\times u_1$, $v_{10}=\beta_0 u_0+\beta_1 u_1+\gamma_1 u_0\times u_1$, $w_{10}=\beta_0 u_0+\beta_1 u_1-\gamma_1 u_0\times u_1$, $v_{01}=\mathcal{P}_{\Pi}(v_{01})+\gamma_2 u_0\times u_1$, and $w_{01}=\mathcal{P}_{\Pi}(w_{01})-\gamma_2 u_0\times u_1$, for some scalars \(\alpha_i\), \(\beta_i\), \(i=0,1\), and \(\gamma_i\), \(i=0,1,2\), such that $\gamma_0\gamma_1>0$.

The spline is $G^1$ continuous if there exist functions $s_1(\xi)$ and $t_1(\xi)$ such that $\tfrac{\partial{\bfm q}}{\partial\eta}(\xi,0)=s_1(\xi) \tfrac{\partial{\bfm p}}{\partial\eta}(\xi,0)+t_1(\xi) \tfrac{\partial{\bfm p}}{\partial\xi}(\xi,0)$. Since
\begin{align*}
\tfrac{\partial{\bfm q}}{\partial\eta}(\xi,0)&=2 \left((1-\xi)\alpha_0+\xi\beta_0\right) u_0+2 \left((1-\xi)\alpha_1+\xi\beta_1\right)u_1-2\left((1-\xi)\gamma_0+\xi\gamma_1\right) u_0\times u_1,\\
\tfrac{\partial{\bfm p}}{\partial\eta}(\xi,0)&=2 \left((1-\xi)\alpha_0+\xi\beta_0\right) u_0+2 \left((1-\xi)\alpha_1+\xi\beta_1\right)u_1+2\left((1-\xi)\gamma_0+\xi\gamma_1\right) u_0\times u_1,\\
\tfrac{\partial{\bfm p}}{\partial\xi}(\xi,0)&=2(1-\xi) u_0+2\xi u_1,
\end{align*}
comparing the coefficients of $u_0\times u_1$ we get $s_1(\xi)=-1$. Hence
\[
4\left((1-\xi)\alpha_0+\xi\beta_0\right)u_0+4\left((1-\xi)\alpha_1+\xi\beta_1\right)u_1= \tfrac{\partial{\bfm q}}{\partial\eta}(\xi,0)-s_1(\xi) \tfrac{\partial{\bfm p}}{\partial\eta}(\xi,0)=t_1(\xi) \tfrac{\partial{\bfm p}}{\partial\xi}(\xi,0)=t_1(\xi)(2(1-\xi) u_0+2\xi u_1)
\]
and comparing the coefficients at $u_0$ and at $u_1$ we obtain
\[
t_1(\xi)=\frac{4\left((1-\xi)\alpha_0+\xi\beta_0\right)}{2(1-\xi)}=\frac{4\left((1-\xi)\alpha_1+\xi\beta_1\right)}{2\xi}.
\]
From this, we conclude that $\alpha_1=\beta_0=0$, $\beta_1=\alpha_0=:\alpha$, and $t_1(\xi)=2\alpha$, so $\mathcal{P}_{\Pi}(v_{00})=\alpha u_0$ and $\mathcal{P}_{\Pi}(v_{10})=\alpha u_1$. Therefore
\[
\tfrac{\partial^2{\bfm q}}{\partial\eta^2}(\xi,0)=
(s_1(\xi))^2 \tfrac{\partial^2{\bfm p}}{\partial\eta^2}(\xi,0)+
2s_1(\xi)t_1(\xi) \tfrac{\partial^2{\bfm p}}{\partial\xi\partial\eta}(\xi,0)+
(t_1(\xi))^2 \tfrac{\partial^2{\bfm p}}{\partial\xi^2}(\xi,0)+
s_2(\xi)\tfrac{\partial{\bfm p}}{\partial\eta}(\xi,0)+
t_2(\xi)\tfrac{\partial{\bfm p}}{\partial\xi}(\xi,0),
\]
where $s_2(\xi)=\tfrac{2 \left((1-2\alpha)\gamma_0+2\alpha\gamma_1-\gamma _2\right)}{(1-\xi)\gamma_0+\xi\gamma_1}$ and $t_2(\xi)=\tfrac{2\alpha \left(\gamma_2-2 (1-2\alpha)\gamma_0-2\alpha\gamma_1\right)}{(1-\xi)\gamma_0+\xi\gamma_1}$. Since $\gamma_0$ and $\gamma_1$ have the same sign, the denominator is nonzero for all $\xi\in[0,1]$. Therefore, the functions $s_2$ and $t_2$ are well defined, hence the spline is also $G^2$ continuous. 

We proved that the first three statements are equivalent. Moreover, the first statement implies the last one. Let us show that the last statement is equivalent to the second one. The normal vector to the tangent plane at $\bfm{p}(\xi,0)$ is 
\begin{align*}
\tfrac{\partial}{\partial\xi}\bfm{p}(\xi,0)\times \tfrac{\partial}{\partial\eta}\bfm{p}(\xi,0)
&=(2(1-\xi) u_0+2\xi u_1)\times (2 \left((1-\xi)\alpha_0+\xi\beta_0\right) u_0+2 \left((1-\xi)\alpha_1+\xi\beta_1\right)u_1+2\left((1-\xi)\gamma_0+\xi\gamma_1\right) u_0\times u_1)\\
&=2\left((1-\xi)\gamma_0+\xi\gamma_1\right)\tfrac{\partial}{\partial\xi}\bfm{p}(\xi,0)\times(u_0\times u_1)+4 \left((1-\xi )^2 \alpha _1+(1-\xi ) \xi  \left(\beta _1-\alpha _0\right)-\xi ^2 \beta _0\right)u_0\times u_1.
\end{align*}
Since $(1-\xi)\gamma_0+\xi\gamma_1\ne 0$ for all $\xi\in[0,1]$, the normal vector of the tangent plane at $\bfm{p}(\xi,0)$ is $\tfrac{\partial}{\partial\xi}\bfm{p}(\xi,0)\times(u_0\times u_1)$ if and only if $(1-\xi )^2 \alpha _1+(1-\xi ) \xi  \left(\beta _1-\alpha _0\right)-\xi ^2 \beta _0=0$, which is equivalent that $\alpha_1=\beta_0=0$ and $\beta_1=\alpha_0$.
\end{proof}

Assume that the patch $\mathbf{p}$ is symmetric with respect to the perpendicular bisector of the line segment joining the vertices $\mathbf{P}_{0,0,2}$ and $\mathbf{P}_{2,0,0}$ and that it induces a $G^1$ spline with its mirror image. Then $\mathcal{P}_{\Pi}(v_{00}) = \alpha u_0$ and $\mathcal{P}_{\Pi}(v_{10}) = (1 - \alpha) u_1$. Hence, the spline is $G^1$ continuous if and only if $\alpha = \tfrac{1}{2}$. This is equivalent to $(v_{00}-\tfrac 1 2 u_0)\times(u_0\times u_1)=0$.

\section{Spherical cap}\label{sec:SphericalCap}

In this section, we construct a $G^1$ approximation of a spherical cap in the form of a spline of triangular patches over a regular spherical polygon. We then show that such approximations over the faces of uniform polyhedra cannot be combined to obtain a $G^1$ approximation of the entire sphere.

Using the above theorem, it is easy to construct a symmetric $G^1$ continuous spline of quadratic triangular patches approximating a spherical cap whose boundary is a regular (curved) $n$-gon for some $n\ge 3$. Without loss of generality, its vertices lie on the unit sphere and the spherical cap can be rotated such that its center is $(0,0,1)$. Hence, the control points of one patch are 
\begin{align*}
b_{2,0,0}&=(0,0,\sqrt{1-r^2}+h)^T,& 
b_{0,2,0}&=\left(r\cos(-\tfrac{\pi}{n}),r\sin(-\tfrac{\pi}{n}),\sqrt{1-r^2}\right)^T,&
b_{0,0,2}&=\left(r\cos(\tfrac{\pi}{n}),r\sin(\tfrac{\pi}{n}),\sqrt{1-r^2}\right)^T,\\
b_{0,1,1}&=(x,0,\sqrt{1-r^2}+z)^T,&
b_{1,1,0}&=\left(s\cos(-\tfrac{\pi}{n}),s\sin(-\tfrac{\pi}{n}),\sqrt{1-r^2}+h\right)^T,&
b_{1,0,1}&=\left(s\cos(\tfrac{\pi}{n}),s\sin(\tfrac{\pi}{n}),\sqrt{1-r^2}+h\right)^T,
\end{align*}
and the remaining patches are obtained by rotating this patch by the angle $\tfrac{2\pi}n$ around the $z$-axis. Hence $u_0=b_{1,0,1}-b_{0,0,2}$, $u_1=b_{2,0,0}-b_{1,0,1}$, $v_{00}=b_{0,1,1}-b_{0,0,2}$, and $v_{10}=b_{1,1,0}-b_{1,0,1}$. Since the adjacent patch is obtained by rotation, the normal to the tangent plane at the point $b_{0,0,2}$ is equal to $u_0\times(u_0\times u_1)$, where $u_0\times u_1=hs(\sin(\tfrac\pi n),-\cos(\tfrac\pi n),0)^T$. Similarly, we can compute the normal to the tangent plane at the point $b_{0,2,0}$. Since $b_{0,1,1}$ lies on both tangent planes, we obtain $z = \tfrac{h}{r-s}\left(r-x \cos\left(\tfrac{\pi}{n}\right)\right)$. Then
\begin{align*}
v_{00}=\frac{r-x \cos \left(\frac{\pi }{n}\right)}{r-s} u_0+\frac{x \sin \left(\frac{\pi }{n}\right)}{h s}u_0\times u_1\quad\text{and}\quad
v_{10}=2 \sin ^2\left(\frac{\pi }{n}\right) u_1+\frac{\sin \left(\frac{2 \pi }{n}\right)}{h} u_0\times u_1.
\end{align*}
By the theorem, the spline is $G^1$ continuous if $({r-s})^{-1}\left(r-x \cos \left(\frac{\pi }{n}\right)\right)=2 \sin ^2\left(\frac{\pi }{n}\right)$, hence
\begin{align}\label{eq:CapParameters}
\begin{split}
x&=\cos^{-1}(\tfrac{\pi}{n})\left(s+(r-s)\cos(\tfrac{2\pi}{n}) \right),\\
z&= \tfrac{h}{r-s}\left(r-x \cos\left(\tfrac{\pi}{n}\right)\right)= \tfrac{h}{r-s}\left(r-\cot(\tfrac{\pi}{n})\left(s+(r-s)\cos(\tfrac{2\pi}{n}) \right)\right).
\end{split}
\end{align}

Below, we present examples of approximations of a spherical cap with base radius \(\tfrac{\sqrt{2}}{2}\) using meshes consisting of 4 to 8 triangles (see \Cref{fig:SphericalCupCurvature}). In each case, all vertices lie on the unit sphere, and each outer boundary curve lies in the plane determined by its endpoints and the origin.

It can be seen that the \(G^1\) continuity condition forces the Gaussian curvature near the common vertex to be very small. On the other hand, increasing the number of triangles results in an increase in the maximum Gaussian curvature. More precisely, the Gaussian curvature ranges from \(0.34\) to \(0.71\) for the mesh with 4 triangles, from \(0.24\) to \(1.14\) for 5 triangles, from \(0.20\) to \(1.76\) for 6 triangles, from \(0.18\) to \(2.30\) for 7 triangles, and from \(0.16\) to \(2.75\) for 8 triangles.
\begin{figure}[htb]
\begin{center}
\includegraphics[width=0.19\textwidth]{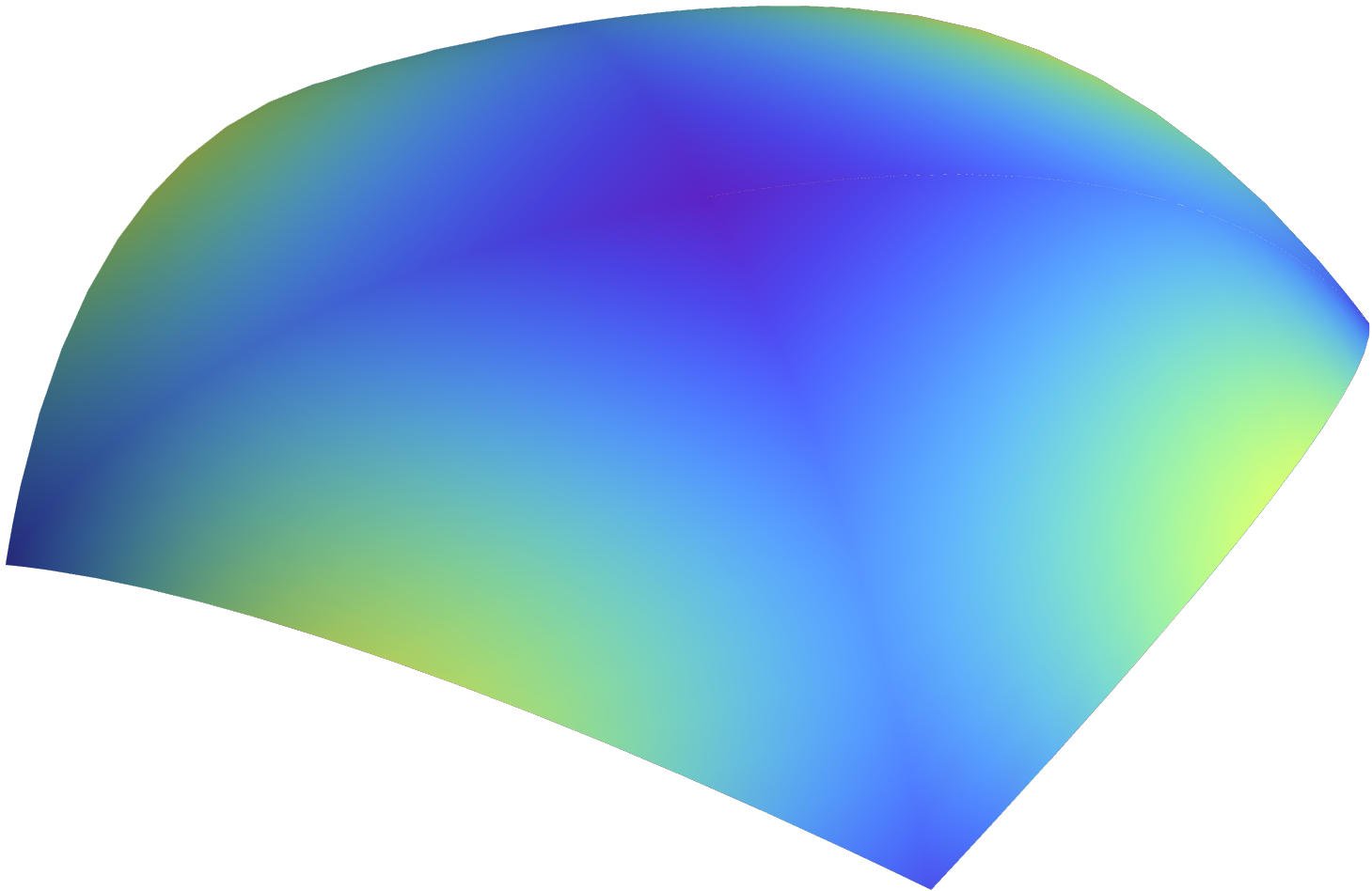}
\includegraphics[width=0.19\textwidth]{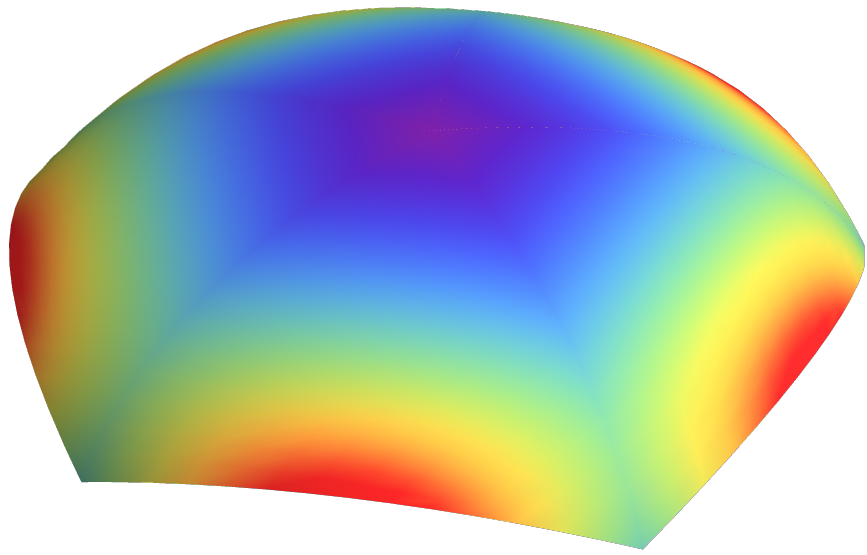}
\includegraphics[width=0.19\textwidth]{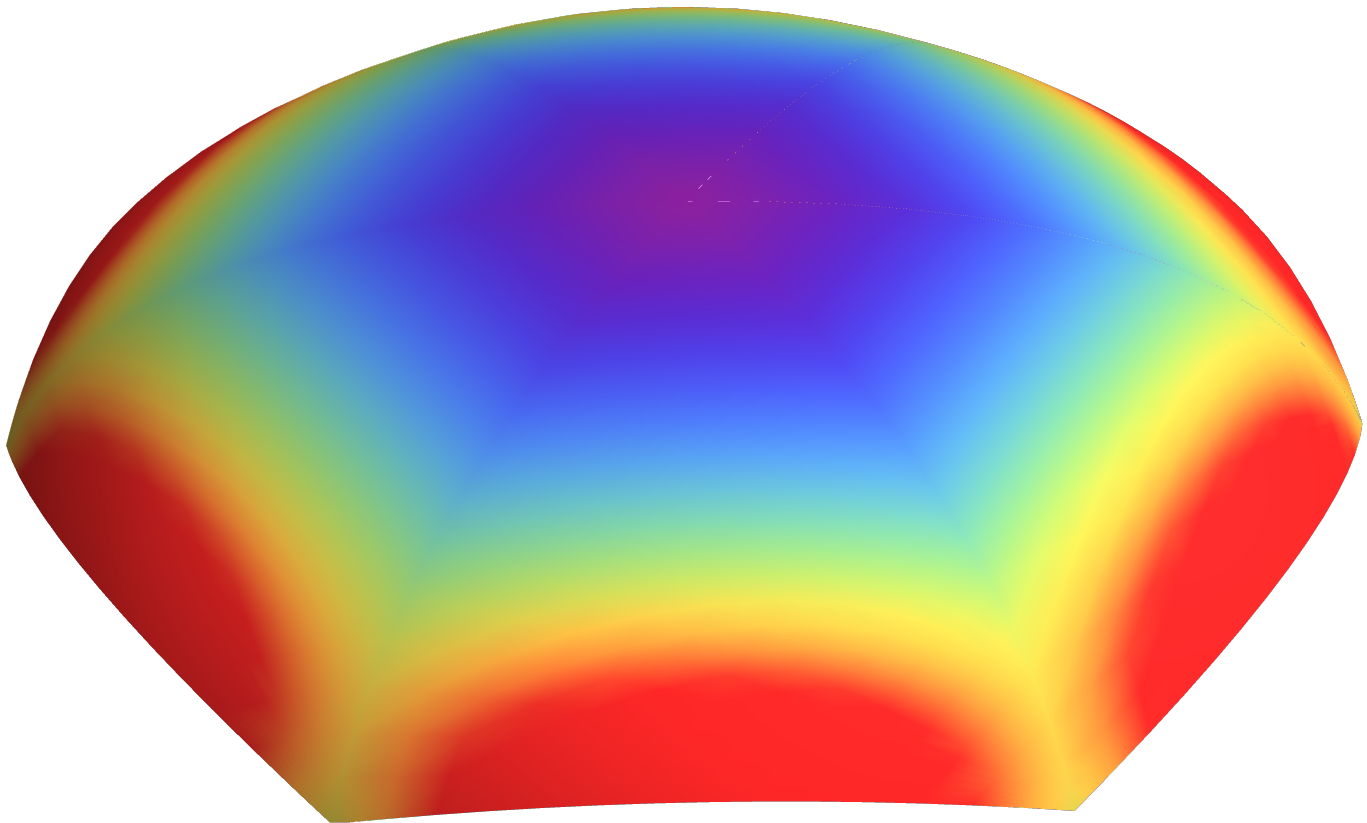}
\includegraphics[width=0.19\textwidth]{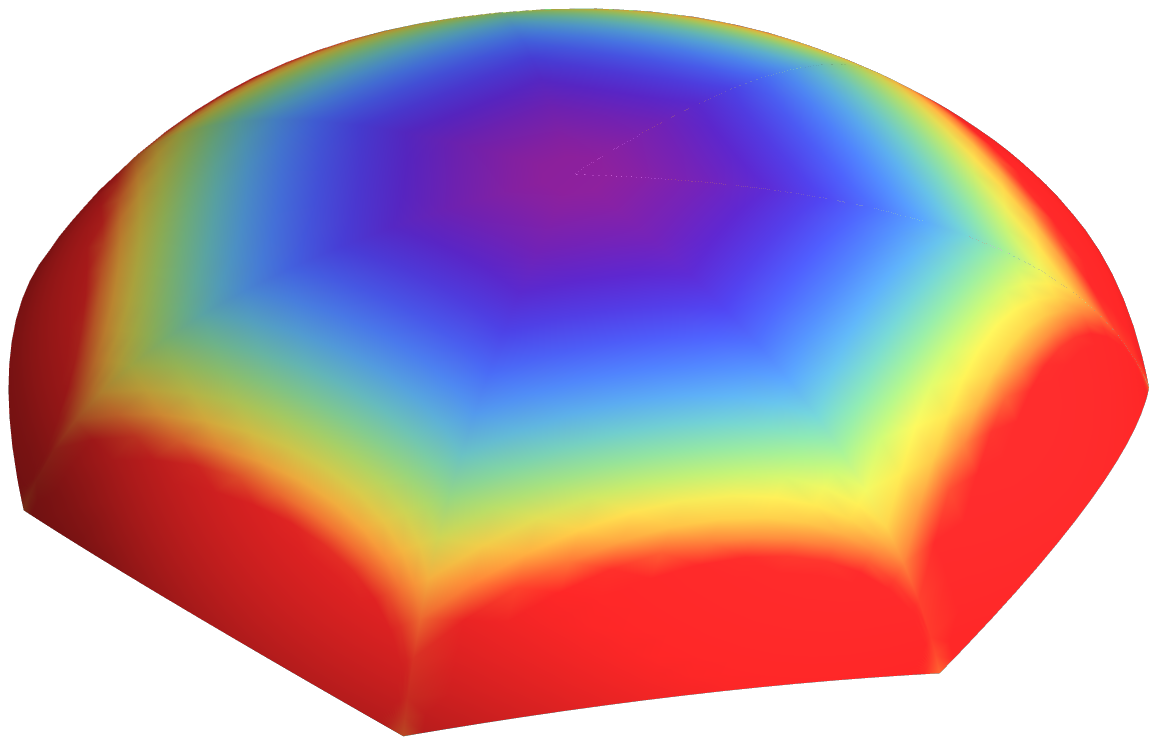}
\includegraphics[width=0.19\textwidth]{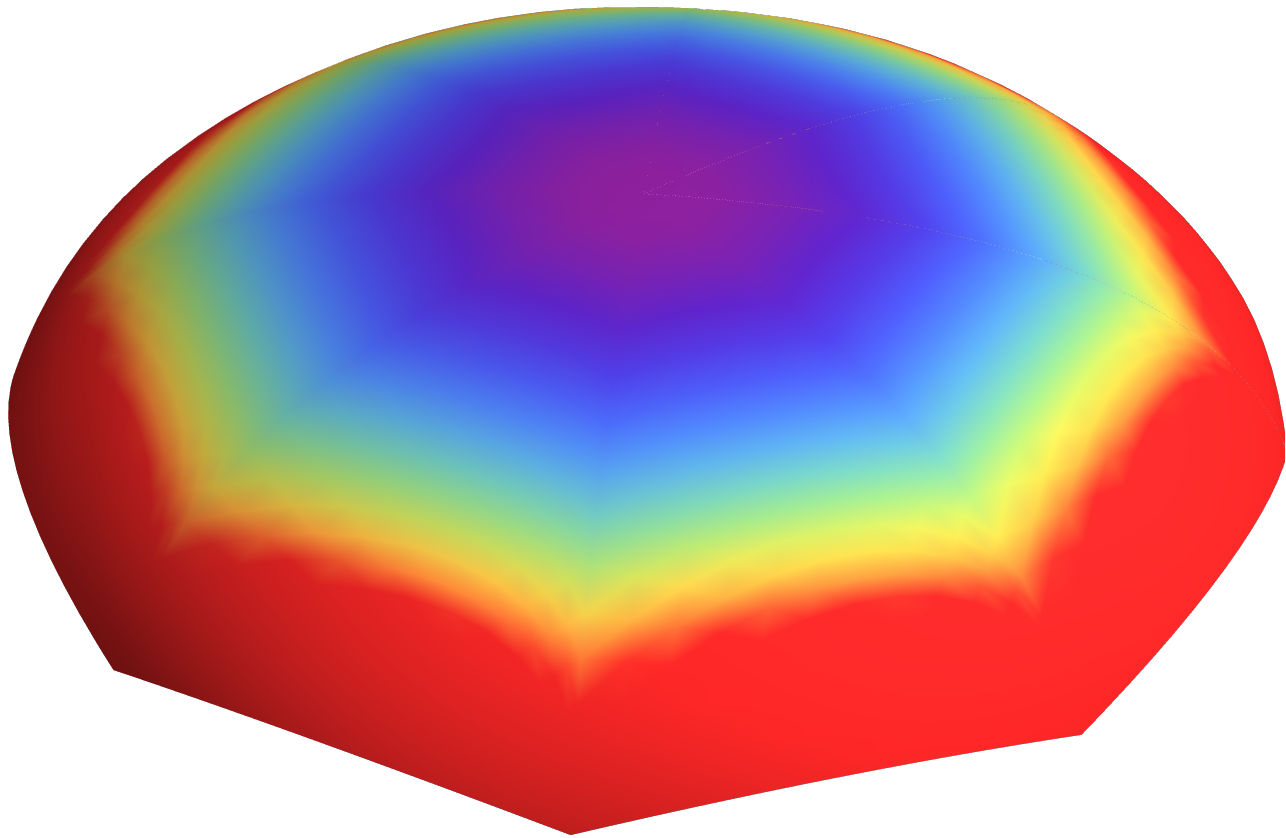}
\caption{Approximations of a spherical cap with base radius \(\tfrac{\sqrt{2}}{2}\), or equivalently, cap height \(1-\tfrac{\sqrt{2}}{2}\). The color of each approximant represents the Gaussian curvature, with dark blue corresponding to \(0.2\) and dark red to \(2.5\).}
\label{fig:SphericalCupCurvature}
\end{center}
\end{figure}
Furthermore, the maximum radial distances of the mesh points from the unit sphere are \(0.093\), \(0.085\), \(0.089\), \(0.095\), and \(0.100\), respectively. Hence, increasing the number of triangles does not improve the quality of the spherical cap approximation.

There is no quadratic $G^1$ spline approximation of the unit sphere induced by any of the three Platonic solids with equilateral spherical triangulations (see \cite{VAVPETIC2022102061}). The next simplest example of a quadratic $G^1$ spline over triangles is obtained by subdividing each regular $n$-gonal face of a uniform polyhedron into $n$ isosceles triangles, as shown on the left of \Cref{fig:SphericalCupTriangulation}.
\begin{figure}[htb]
\begin{center}
\includegraphics[width=0.3\textwidth]{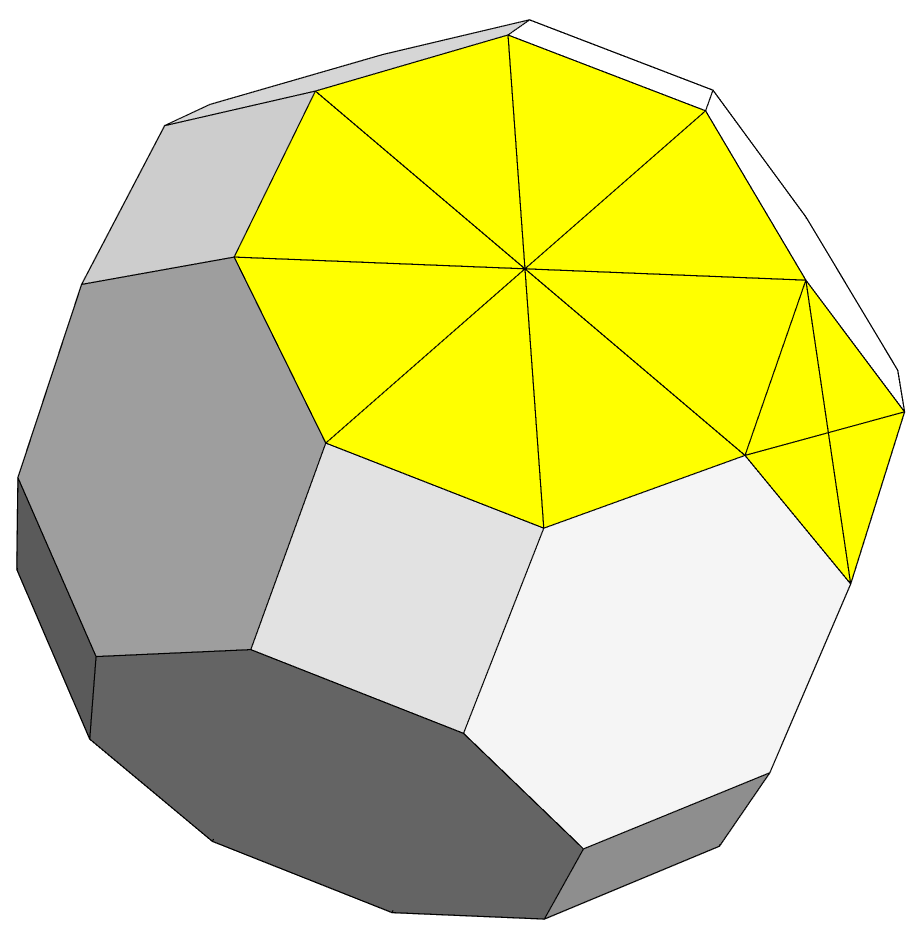}
\hskip2cm
\includegraphics[width=0.3\textwidth]{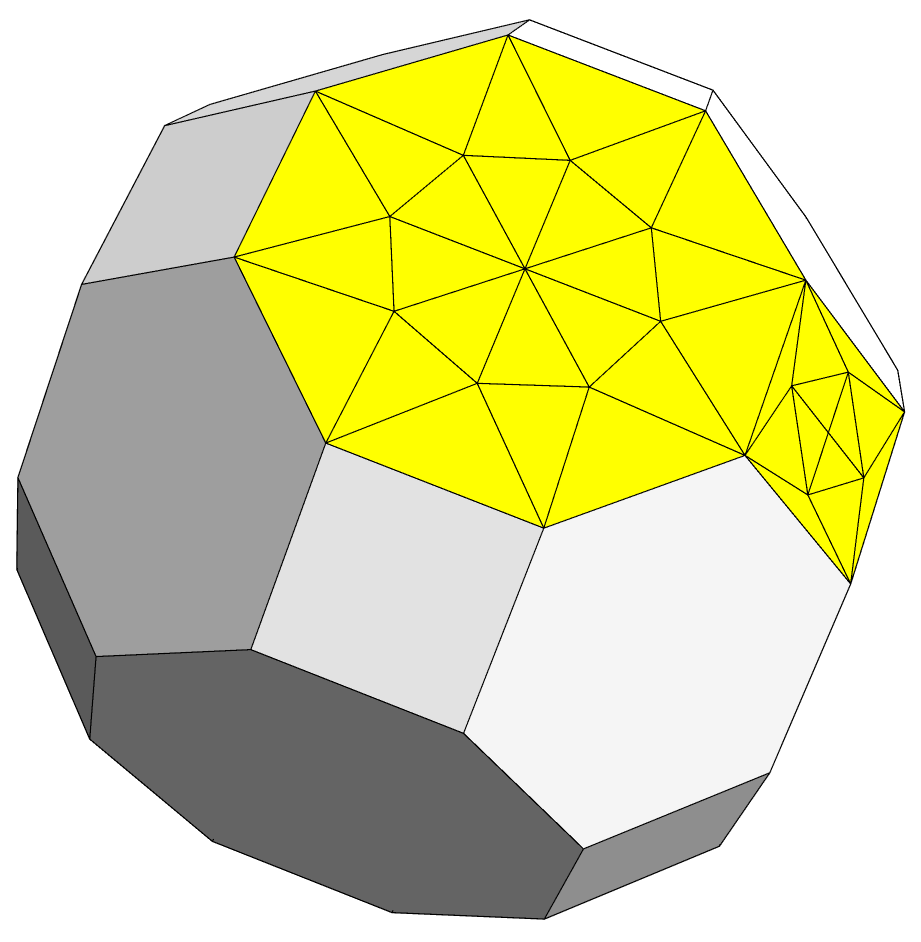}
\caption{Two ways to triangulate faces of a uniform solid to approximate spherical caps.}
\label{fig:SphericalCupTriangulation}
\end{center}
\end{figure}
We would like to construct a spline such that, along every edge of the uniform polyhedron, both the boundary curve and the associated tangent planes are defined identically on the two patches sharing that edge. This is made possible by the above theorem. We place the control point $\bfm{b}_{0,1,1}$ in the plane determined by the origin and the vertices $\bfm{b}_{0,2,0}$ and $\bfm{b}_{0,0,2}$, obtaining $h=\sqrt{1-r^2} (2 s-r)r^{-1}\left(1+\cos \left(\tfrac{2 \pi }{n}\right)\right)^{-1}$. We then require that this patch, together with its mirror image across the above-mentioned plane, form a $G^1$ spline. By the theorem, the tangent planes at the boundary points depend only on the boundary curve itself. From the remark following the theorem, we obtain that $(\bfm{b}_{1,1,0}-\bfm{b}_{0,2,0})-\tfrac  1 2(\bfm{b}_{0,1,1}-\bfm{b}_{0,2,0})$ is parallel to $(\bfm{b}_{0,1,1}-\bfm{b}_{0,2,0})\times(\bfm{b}_{0,0,2}-\bfm{b}_{0,1,1})$ which is impossible.

\section{Triangulation of faces with an additional ring of triangles}\label{sec:Triangulation}
As we have seen, triangulating an \(n\)-gon using only \(n\) isosceles triangles is not sufficient to construct a \(G^1\) spline. Therefore, we add a ring of \(2n\) triangles around them as follows. Inside a regular $n$-gon with vertices $V_i$, $i=1,\ldots,n$, we inscribe a smaller regular $n$-gon with vertices $W_i$, $i=1,\ldots,n$, having the same center $C$ as the larger polygon and is rotated by an angle of $\tfrac{\pi}n$ relative to it. The resulting triangulation consists of three types of isosceles triangles, namely $V_iV_{i+1}W_i$, $W_iW_{i+1}V_{i+1}$, and $W_iW_{i+1}C$, $i=1,\ldots,n$, where $V_{n+1}=V_1$ and $W_{n+1}=W_1$ (see \Cref{fig:triangulation2} and the right-hand image in \Cref{fig:SphericalCupTriangulation}).
\newcommand{\Triangulacija}[1]{
\def\n{#1};
\def\r{3};
\def\rr{3.3};
\def\s{1};
\def\ss{1.3};
\draw (180/\n:\r)\foreach \x in {1,...,\n}{ -- (\x*360/\n+180/\n:\r)};
\draw (0:\s)\foreach \x in {1,...,\n}{ -- (\x*360/\n:\s)};
\foreach \x in {1,...,\n}{\draw (\x*360/\n-3*180/\n:\rr) node{$V_\x$};}
\foreach \x in {1,...,\n}{\draw (\x*360/\n-360/\n:\ss) node[anchor=center, shift={(\x*360/\n-360/\n:0.1)}]{$W_\x$};}
\foreach \x in {1,...,\n}{\draw (\x*360/\n-180/\n:\r)--(\x*360/\n:\s)--(\x*360/\n+180/\n:\r);}
\foreach \x in {1,...,\n}{\draw (0,0)--(\x*360/\n:\s);}
\draw (180/\n:0.3)node{$C$};}
\begin{figure}[h!]
\begin{center}
\begin{tikzpicture}[line join=round, line cap=round, thick,scale=0.8]
\Triangulacija{3}
\begin{scope}[xshift=5.5cm]
\Triangulacija{4}    
\end{scope}
\begin{scope}[xshift=12.5cm]
\Triangulacija{5}
\end{scope}
\end{tikzpicture}
\caption{Triangulations of square, triangular, and pentagonal faces}\label{fig:triangulation2}
\end{center}
\end{figure}
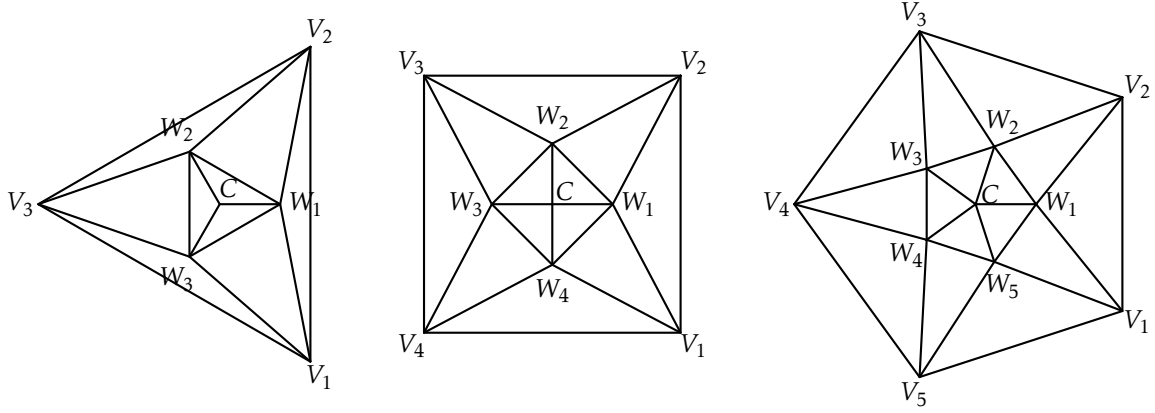

It suffices to construct B\'ezier patches over the triangles \(\Delta_1 := V_1 W_1 V_2\), \(\Delta_2 := V_1 W_1 W_n\), and \(\Delta_3 := W_1 C W_n\); a spline over the entire face is then obtained by applying appropriate rotations to these representative patches. Let $\bfm{b}^m_{i,j,k}$ denote the control points and $\bfm{p}_m(\xi,\eta)$ the parametrization of the patch associated with triangle $\Delta_m$, $m = 1,2,3$. 
Without loss of generality, we may assume that the face is parallel to $xy$-plane, that its centroid lies on the positive $z$-axis, and that it is oriented as shown in \Cref{fig:triangulation2}, so the vertex control points of the patch over the triangle $\Delta_1$ are 
\begin{equation*}
\bfm{b}^1_{0,0,2} =\left(rc, - r\sqrt{1-c^2}, \sqrt{1-r^2} \right)^T, \quad
\bfm{b}^1_{0,2,0} =(x_1,0,z_1)^T, \quad
\bfm{b}^1_{2,0,0} =\left(rc,  r\sqrt{1-c^2}, \sqrt{1-r^2}\right)^T, 
\end{equation*}
where $c = \cos \frac{\pi}{n}$ and $x_1,z_1>0$. To ensure $G^0$ continuity between patches defined over adjacent faces, the control point $\bfm{b}^1_{1,0,1}$ must lie in the plane determined by the points $\bfm{b}^1_{0,0,2}$, $\bfm{b}^1_{2,0,0}$, and the origin. For $G^1$ continuity at $\bfm{b}^1_{0,0,2}$, all directional vectors at this point must be orthogonal to its position vector. A similar condition holds at $\bfm{b}^1_{2,0,0}$. Therefore, the control point $\bfm{b}^1_{1,0,1}$ is the intersection of three planes: the previously mentioned plane and the tangent planes to the sphere $\mathcal{S}$ at both $\bfm{b}^1_{0,0,2}$ and $\bfm{b}^1_{2,0,0}$. Furthermore, the control points $\bfm{b}^1_{0,1,1}$ and $\bfm{b}^1_{1,1,0}$ lie in the tangent planes to the sphere $\mathcal{S}$ at $\bfm{b}^1_{0,0,2}$ and $\bfm{b}^1_{2,0,0}$, respectively. Consequently, we set
\begin{align*}
\bfm{b}^1_{0,1,1} = \left(x_2, -y_2 , \tfrac{1 - r(c x_2 + y_2\sqrt{1-c^2} ) }{\sqrt{1-r^2}} \right)^T,\quad
\bfm{b}^1_{1,0,1} =\left(\tfrac{c r}{1-(1-c^2) r^2}, 0, \tfrac{\sqrt{1-r^2}}{1-(1-c^2) r^2}\right)^T,\quad
\bfm{b}^1_{1,1,0} = \left(x_2, y_2 , \tfrac{1 - r(c x_2 + y_2\sqrt{1-c^2} ) }{\sqrt{1-r^2}}\right)^T, 
\end{align*}
where $x_2,y_2>0$. To define the control points associated with the triangle $\Delta_2$, let $\mathcal{R}_{\Sigma}$ denote the reflection across the plane $\Sigma$ containing the $z$-axis and the control point $\bfm{b}^1_{0,0,2}$. The equation of the plane $\Sigma$ is $\sqrt{1-c^2} x+c y = 0$. By symmetry, the control points are
\begin{align*}
\bfm{b}^2_{0,0,2} = \bfm{b}^1_{0,0,2},\quad\bfm{b}^2_{0,1,1} = \bfm{b}^1_{0,1,1},\quad
\bfm{b}^2_{0,2,0} =\bfm{b}^1_{0,2,0},\quad \bfm{b}^2_{1,0,1} = \mathcal{R}_{\Sigma}(\bfm{b}^1_{0,1,1}),\quad
\bfm{b}^2_{1,1,0} = \left(x_3, - \tfrac{\sqrt{1-c^2}}{c} x_3 , z_3 \right)^T,\quad
\bfm{b}^2_{2,0,0} = \mathcal{R}_{\Sigma}(\bfm{b}^1_{0,2,0}), 
\end{align*}
where $x_3,z_3>0$. For the control points associated with the triangle $\Delta_3$, we set
\begin{align*}
\bfm{b}^3_{0,0,2} =\bfm{b}^2_{0,2,0},\quad 
\bfm{b}^3_{0,1,1} =\bfm{b}^2_{1,1,0},\quad 
\bfm{b}^3_{0,2,0} = \bfm{b}^2_{2,0,0},\quad
\bfm{b}^3_{1,0,1} = (x_4,0,z_4)^T, \quad
\bfm{b}^3_{1,1,0} =\mathcal{R}_{\Sigma}(\bfm{b}^3_{1,0,1}),\quad
\bfm{b}^3_{2,0,0} =(0,0,z_4)^T, 
\end{align*}
where $x_4,z_4>0$. We have taken into account the fact that, for $G^1$ continuity at $\bfm{b}^3_{2,0,0}$, all directional vectors at this point must be orthogonal to its position vector. The remaining eight unknown parameters $x_1,z_1,x_2,y_2,x_3,z_3,x_4$ and $z_4$ are determined by $G^1$ continuity conditions along the four common edges of adjacent patches.

Let us start with the patch associated with triangle $\Delta_1$ and the boundary curve $\bfm{p}_1(\xi,0)$. As at the end of the previous section, we obtain the condition that the normal vector $(\bfm{b}^1_{1,0,1}-\bfm{b}^1_{0,0,2})\times(\bfm{b}^1_{2,0,0}-\bfm{b}^1_{1,0,1})$ is parallel to the vector $(\bfm{b}^1_{0,1,1}-\bfm{b}^1_{0,0,2})-\tfrac 1 2(\bfm{b}^1_{1,0,1}-\bfm{b}^1_{0,0,2})$ which yields $y_2 = \frac{1}{2}\sqrt{1-c^2}r$. 
For $G^1$ continuity between $\Delta_3$ and its mirror image with respect to the plane through the points $\bfm{b}^3_{0,0,2}$, $\bfm{b}^3_{2,0,0}$, and the coordinate origin, we can use equation (\ref{eq:CapParameters}). This yields $x_3=(2c^2-1)x_1+2(1-c^2)x_4$ and $z_3 = (2c^2-1) z_1+2(1-c^2) z_4$.
The tangent planes at the point $\bfm{p}_1(0,1)=\bfm{p}_2(0,1)=\bfm{p}_3(0,0)$ coincide. Because of the symmetry of the patch over $\Delta_1$, the second coordinate of the normal vector to the tangent plane is zero. The same therefore holds for the patch over $\Delta_2$, and we get $z_4=\tfrac{1}{2 \sqrt{1-r^2}(x_1-x_2)}\left((x_1-x_4) \left(2-\left(1-c^2\right) r^2-2 c r x_2\right)+2\sqrt{1-r^2} (x_4 z_1-x_2 z_1)\right)$.
Similarly, the $G^1$ continuity along the common edge $\bfm{p}_2(\eta,1-\eta)=\bfm{p}_3(0,\eta)$ yields the condition $\det(\tfrac{\partial}{\partial\xi}\bfm{p}_3(0,\eta),\tfrac{\partial}{\partial\eta}\bfm{p}_3(0,\eta),\tfrac{\partial}{\partial\xi}\bfm{p}_2(\eta,1-\eta))=0$. Using the above expressions for $y_2$, $x_3$, $z_3$ and $z_4$, we obtain that the determinant is of the form $k_1(x_1,x_2,x_4,z_1)\eta(1-\eta)$. The $G^1$ continuity along the common edge $\bfm{p}_1(0,\eta)=\bfm{p}_2(0,\eta)$ yields the condition $\det(\tfrac{\partial}{\partial\xi}\bfm{p}_1(0,\eta),\tfrac{\partial}{\partial\eta}\bfm{p}_1(0,\eta),\tfrac{\partial}{\partial\xi}\bfm{p}_2(0,\eta))=0$, where the determinant has a form $(k_2(x_1,x_2,x_4,z_1)\eta+k_3(x_1,x_2,x_4,z_1))\eta(1-\eta)$. Hence, for the remaining four variables, we obtain three rational equations $k_i(x_1,x_2,x_4,z_1)=0$, $i=1,2,3$. This system admits several one-parameter families of solutions, among which the following appears to be the most suitable:
\begin{align*}
z_1 = \tfrac{2-2 c r x_1+ \left(c^2-1\right) r^2}{2 \sqrt{1-r^2}},\quad
x_2 = \tfrac{cr +x_1}{2}, \quad
y_2 = \tfrac{r \sqrt{1-c^2}}{2},\quad
x_3= c^2x_1,\quad
z_3 = \tfrac{2-2 c^3 r x_1+ \left(c^2-1\right) r^2}{2 \sqrt{1-r^2}}, \quad
x_4 = \tfrac{x_1}{2},\quad
z_4 = \tfrac{2- c r x_1+ \left(c^2-1\right) r^2}{2 \sqrt{1-r^2}}.
\end{align*}
To ensure that the normal vector at $\bfm{p}_1(0,1)$ is nonzero, we must have $x_1 < x_2$. This is the only condition under which a normal vector can degenerate. Therefore, $0< x_1 < cr$. In the limiting case $x_1=cr$, we obtain a spline for which the normal vector of the tangent plane vanishes only at the point $\bfm{p}_1(0,1)=\bfm{p}_2(0,1)=\bfm{p}_3(0,0)$. Note that the normal vectors to the tangent planes
\begin{align*}
\bfm{n}_1(\xi,\eta)=\tfrac{\partial{\bfm{p}_1}}{\partial\xi}(\xi,\eta)\times\tfrac{\partial{\bfm{p}_1}}{\partial\eta}(\xi,\eta)&=\tfrac{2 \sqrt{1-c^2} r}{(1-\left(1-c^2\right) r^2)\sqrt{1-r^2}}\left((c r-x_1)+\left(1-c^2\right) r^2 (x_1-c r\eta )\right) \left(c r,\sqrt{1-c^2} r (\eta +2 \xi-1),\sqrt{1-r^2}\right)^T,\\  
\bfm{n}_2(\xi,\eta)=\tfrac{\partial{\bfm{p}_2}}{\partial\xi}(\xi,\eta)\times\tfrac{\partial{\bfm{p}_2}}{\partial\eta}(\xi,\eta)&= \tfrac{2 \sqrt{1-c^2}}{\sqrt{1-r^2}} x_1 (c x_1-r) \left(c r \left(1-2 \left(1-c^2\right) \xi \right),\sqrt{1-c^2} r \left(\eta +\left(1-2 c^2\right)\xi -1\right),\sqrt{1-r^2}\right)^T,\\    
\bfm{n}_3(\xi,\eta)=\tfrac{\partial{\bfm{p}_3}}{\partial\xi}(\xi,\eta)\times\tfrac{\partial{\bfm{p}_3}}{\partial\eta}(\xi,\eta)&= 2 \tfrac{c \sqrt{1-c^2}}{\sqrt{1-r^2}} x_1^2 \left(c r \left(1-2 \left(1-c^2\right) \eta -\xi\right),-2 c^2 \sqrt{1-c^2} r \eta,\sqrt{1-r^2}\right)^T,  
\end{align*}
are nonzero, since their last coordinates are nonzero, $r<1$ and $x_1<c r$.

The construction of the patches thus depends on the parameters $c$ and $r$, which are determined by the geometry of the polyhedron $\mathcal{P}$, as well as on the choice of the free parameter $x_1$. When $x_1$ is close to $cr$, the patches whose boundary curves lie above the edges of the Archimedean solid become thin, while the overall shape of the spline remains close to that of the sphere. As $x_1$ decreases, however, the spline becomes flatter near the vertices of the solid and develops bumps above the face centers, as illustrated in \Cref{fig:differentx1} for the icosahedron.
\begin{figure}[htb]
\begin{center}
\includegraphics[width=0.2\textwidth]{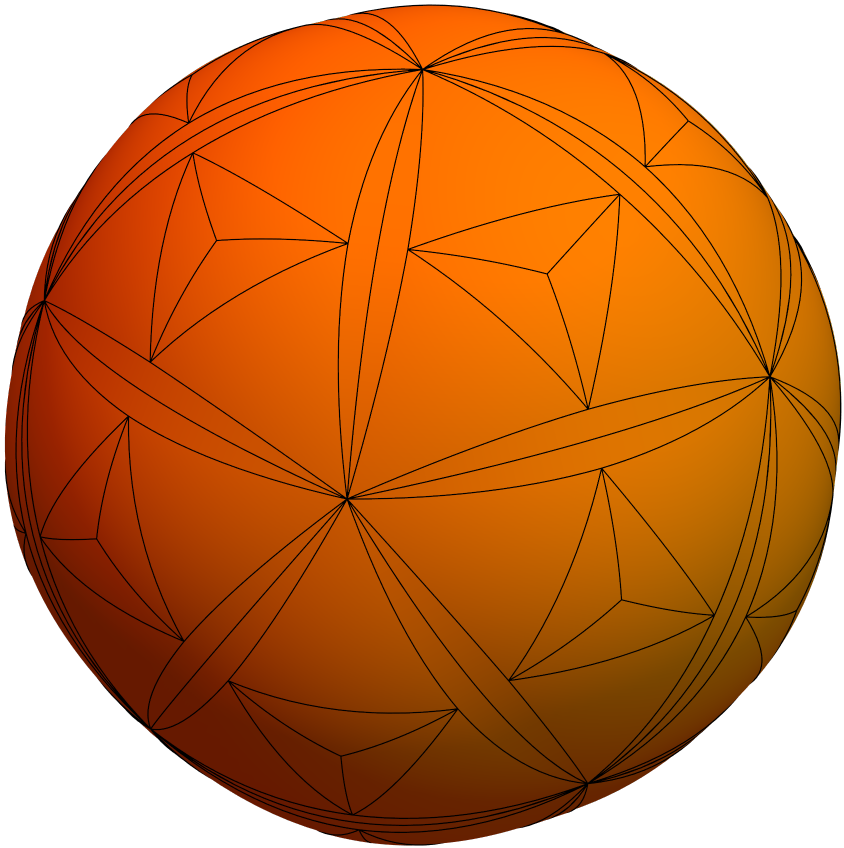}\hfil
\includegraphics[width=0.2\textwidth]{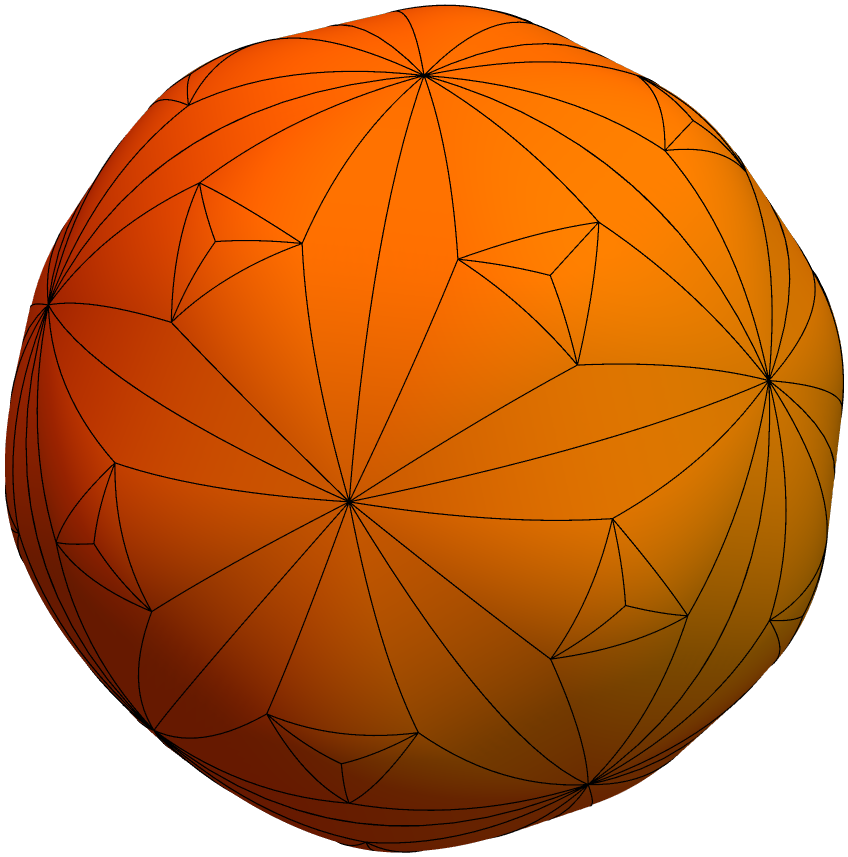}\hfil
\includegraphics[width=0.2\textwidth]{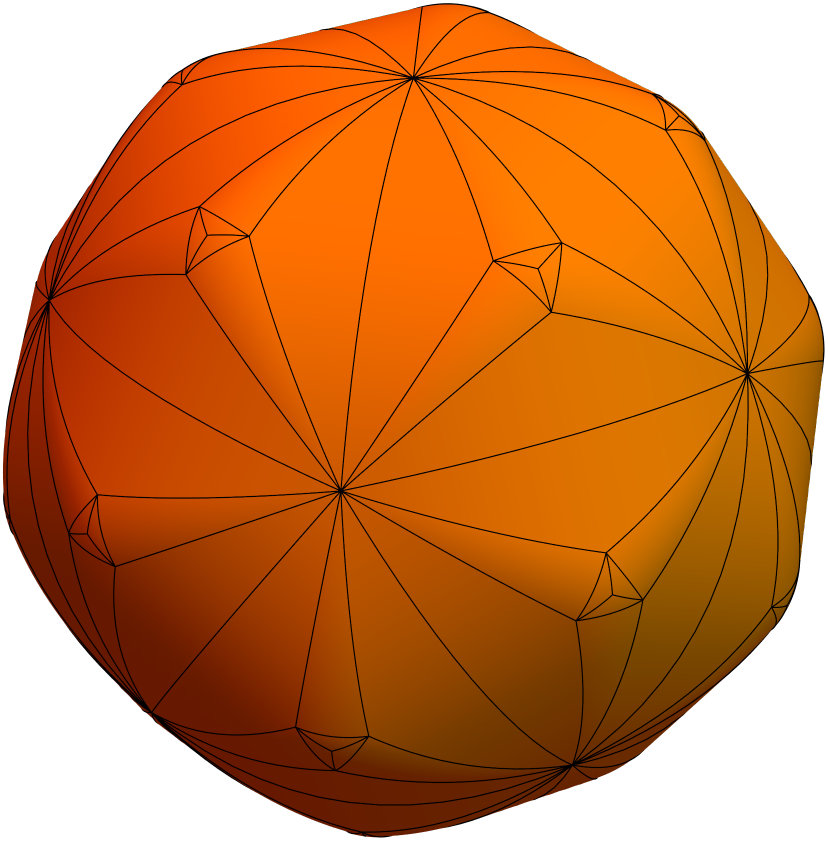}
\caption{\(G^1\) approximants of the icosahedron for \(x_1=0.3\), \(x_1=0.2\), and \(x_1=0.1\); the limiting value is \(cr\approx 0.3035\). The maximum radial distances between the approximants and the unit sphere are $0.055$, $0.095$, and $0.135$ for $x_1=0.3$, $x_1=0.2$, and $x_1=0.1$, respectively. The corresponding maximum Gaussian curvatures are $1.62$, $3.65$, and $14.59$.}
\label{fig:differentx1}
\end{center}
\end{figure}
From the figures, we observe that as the central triangles, i.e., those having one vertex above the center of a face of a polyhedron, become smaller, the quality of the approximation deteriorates both in terms of distance and curvature. Hence, the best regular approximants are obtained when the value of $x_1$ is as close as possible to the value $cr$. Figure \ref{fig:UniformSolids} illustrates all examples where $\mathcal{P}$ is a Platonic or Archimedean solid for $x_1=cr$. Table \ref{tab:errors} presents limiting values of the radial errors and Gaussian curvature ranges for the same polyhedra as $x_1$ approaches $cr$.
\begin{figure}[htb]
\begin{center}
\includegraphics[width=0.16\textwidth]{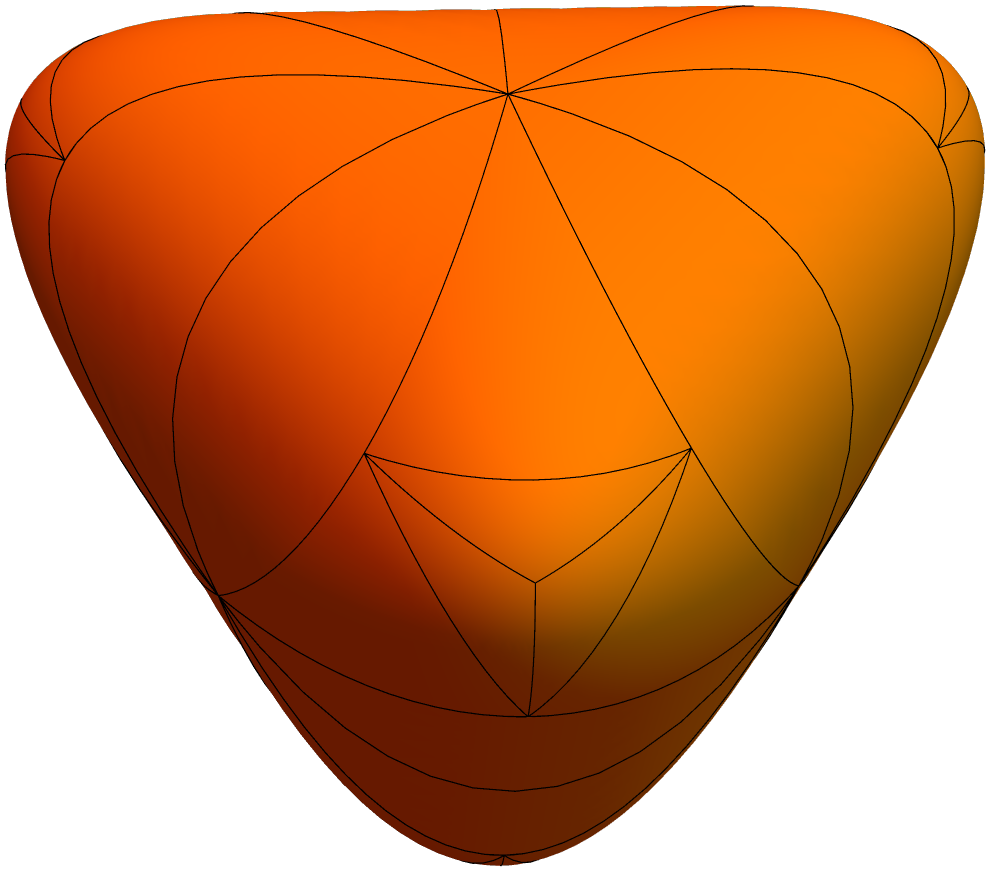}
\includegraphics[width=0.16\textwidth]{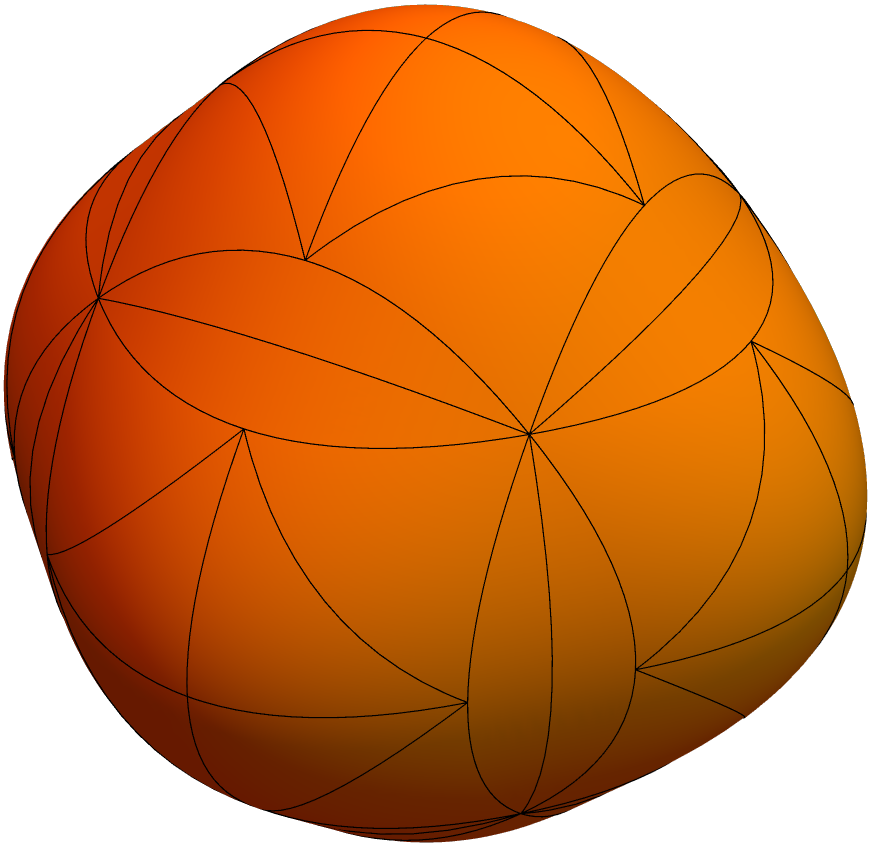}
\includegraphics[width=0.16\textwidth]{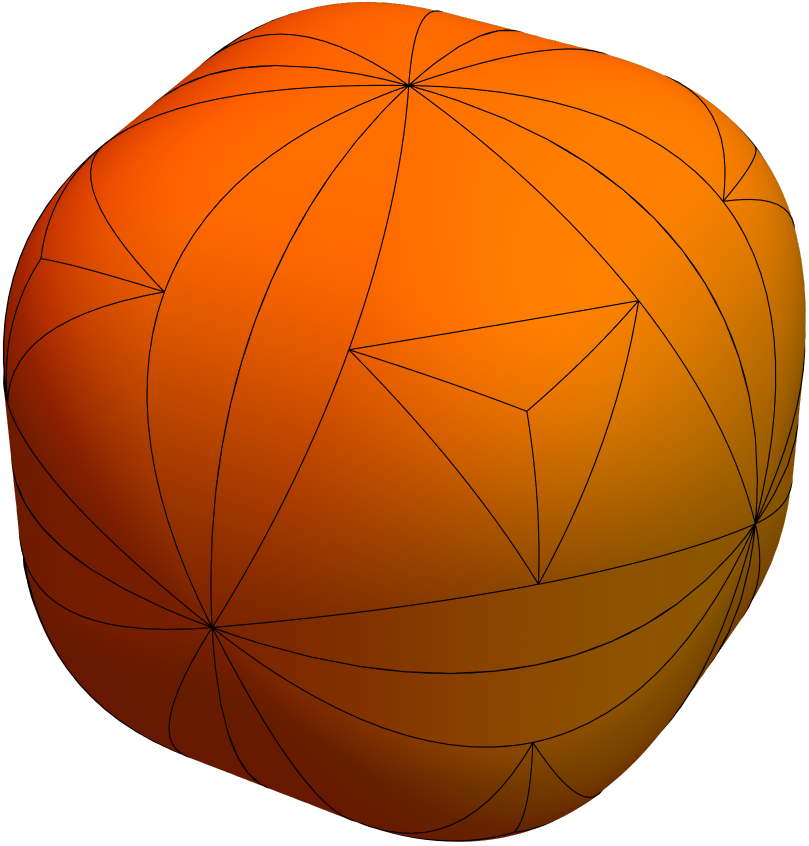}
\includegraphics[width=0.16\textwidth]{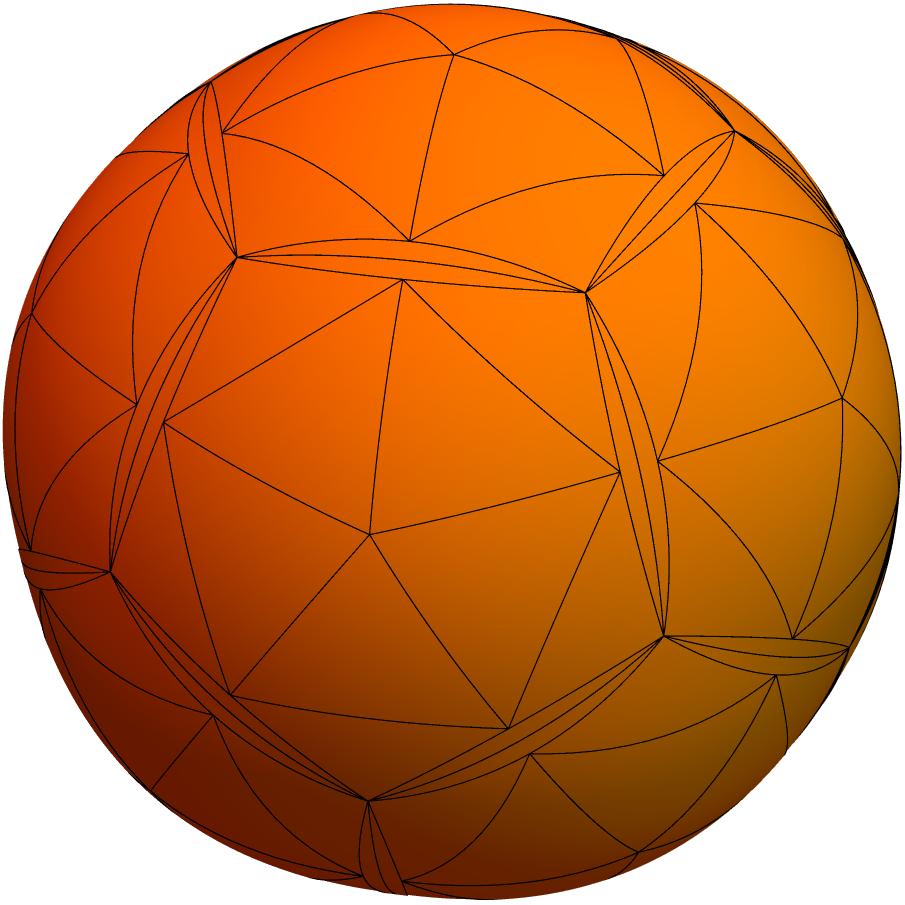}
\includegraphics[width=0.16\textwidth]{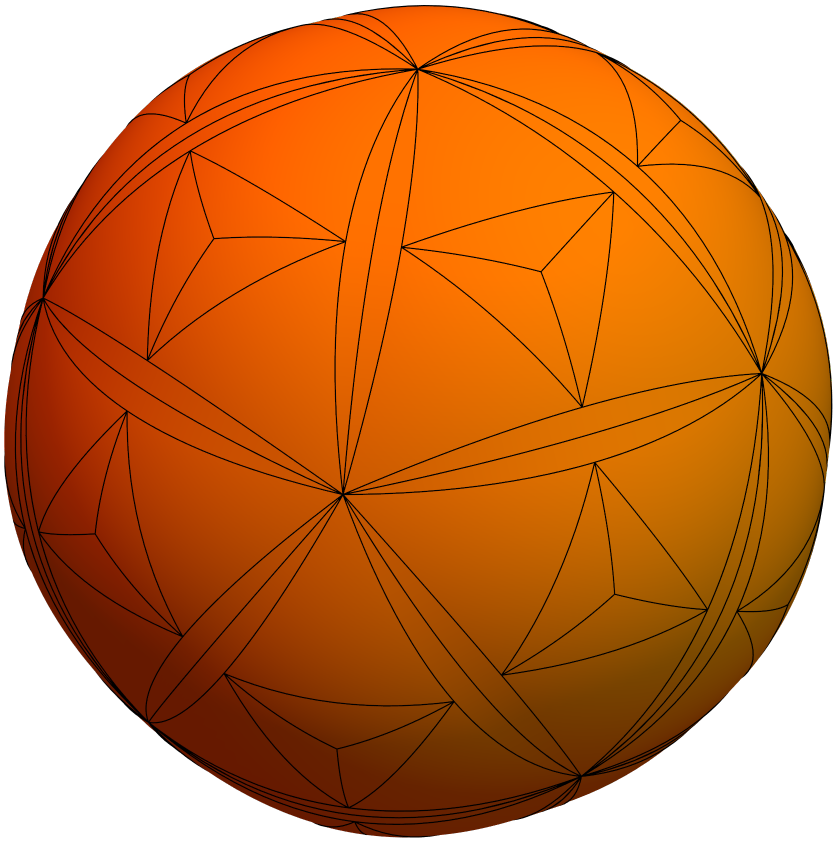}
\includegraphics[width=0.16\textwidth]{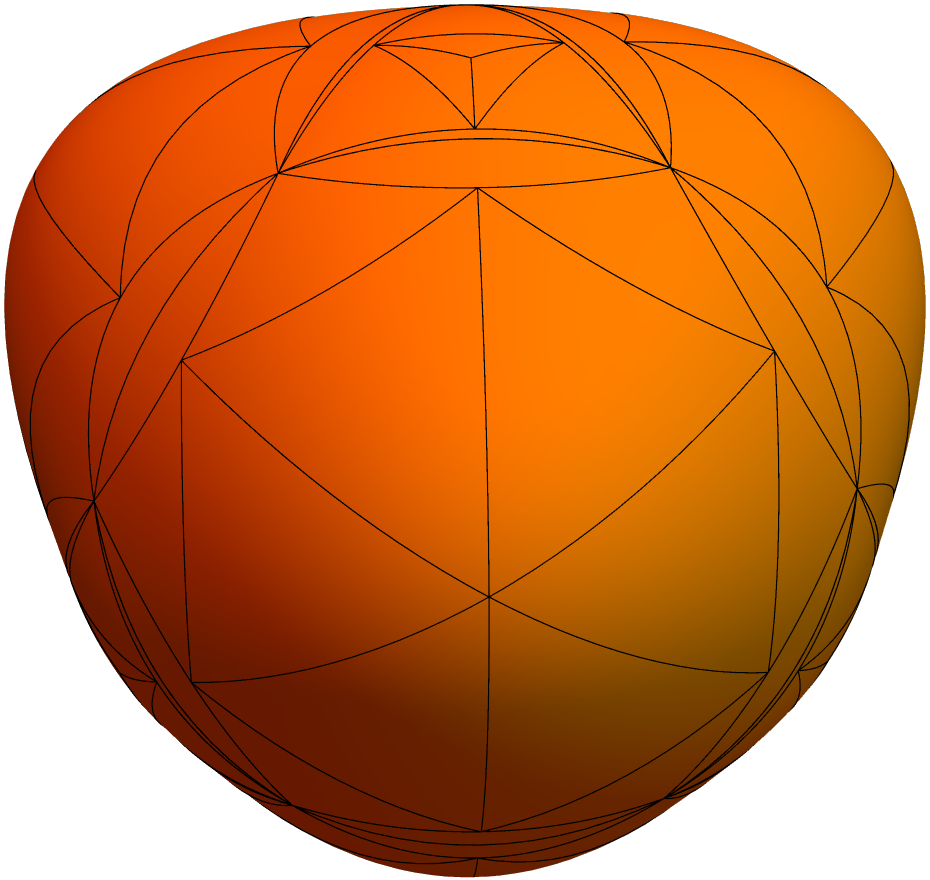}
\includegraphics[width=0.16\textwidth]{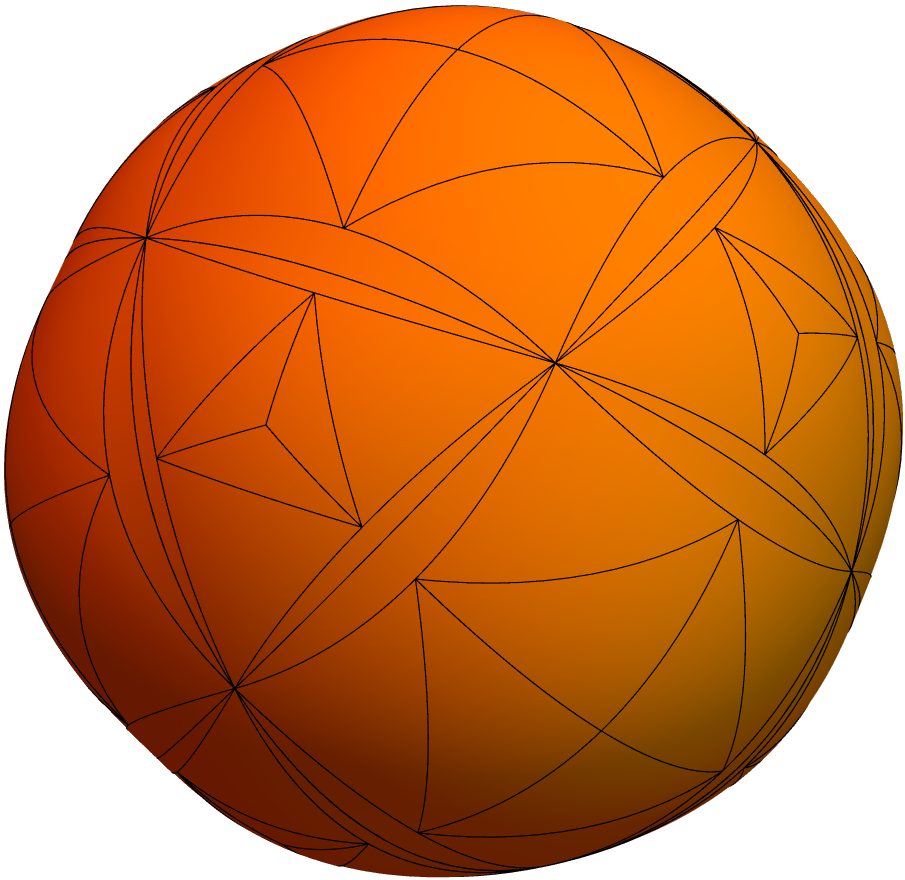}
\includegraphics[width=0.16\textwidth]{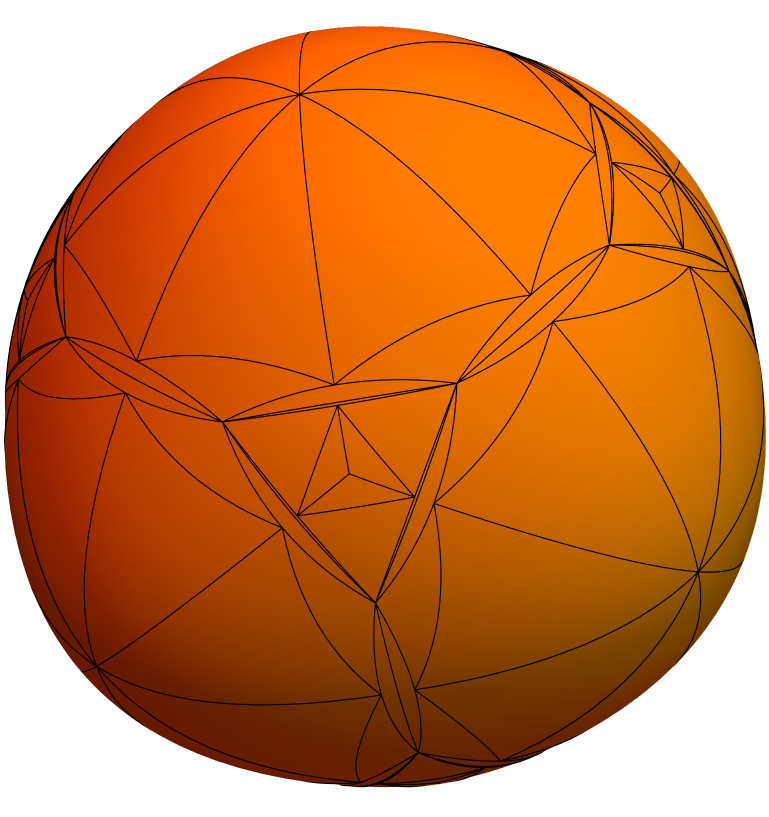}
\includegraphics[width=0.16\textwidth]{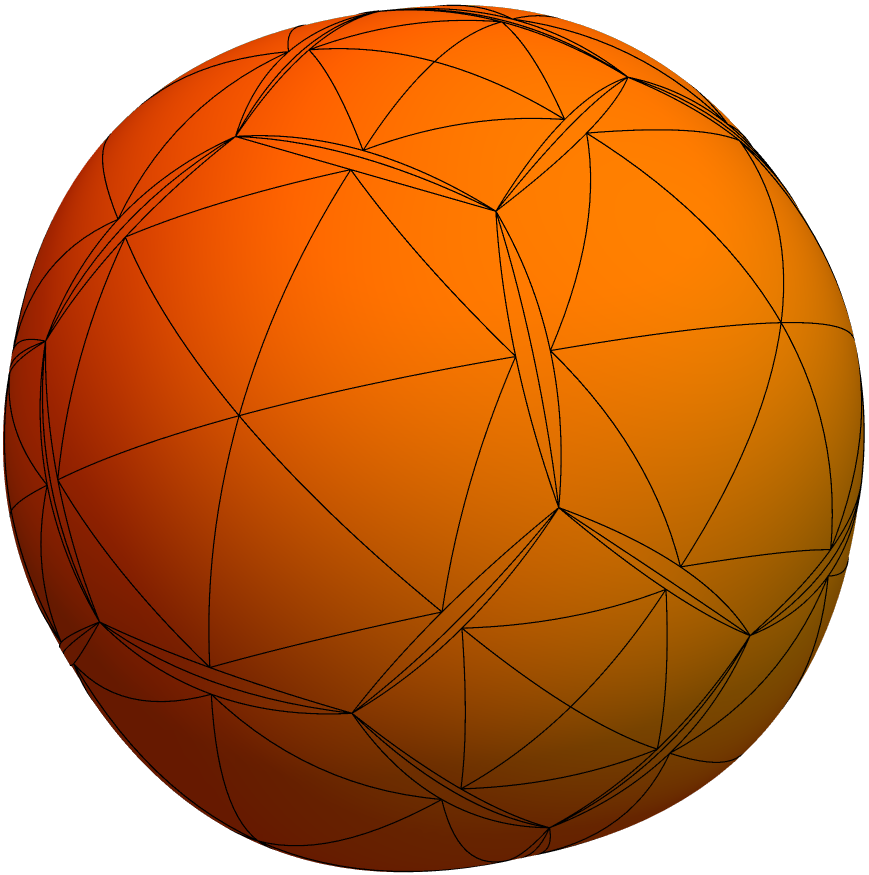}
\includegraphics[width=0.16\textwidth]{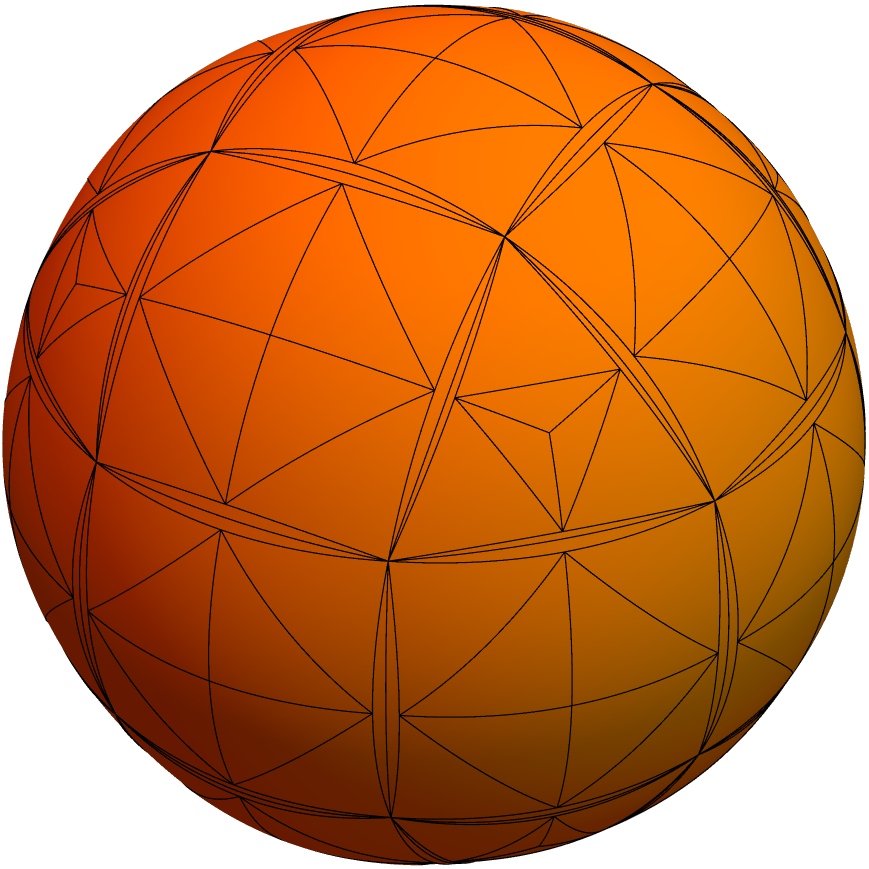}
\includegraphics[width=0.16\textwidth]{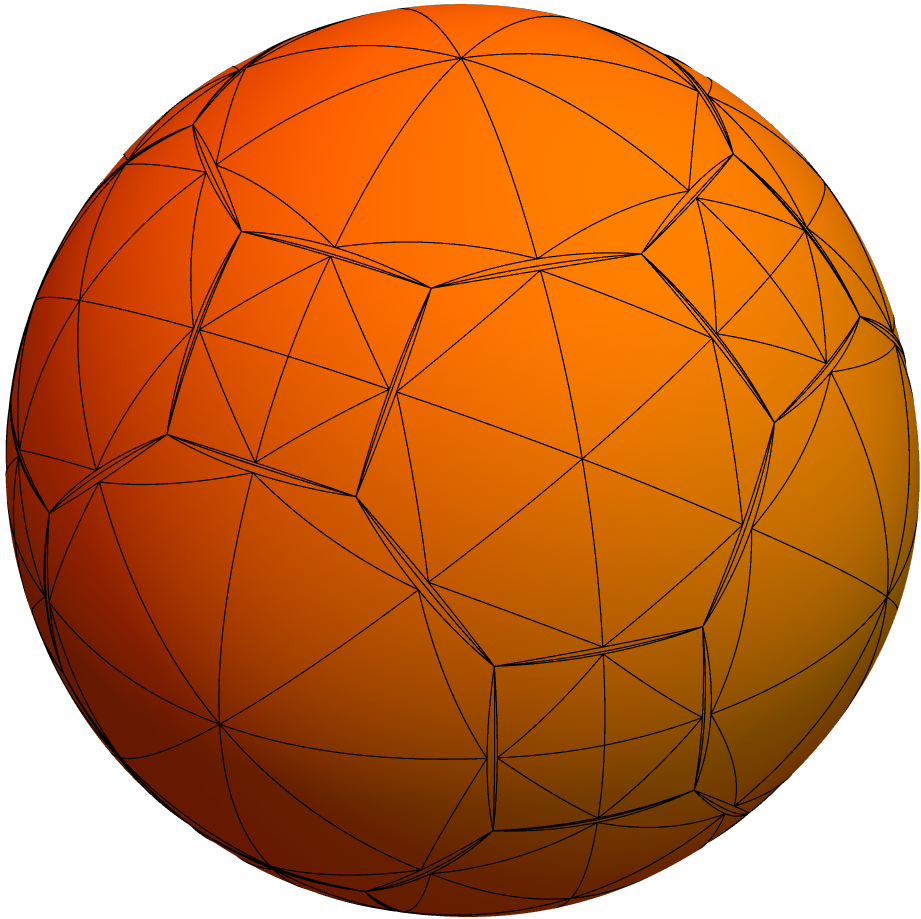}
\includegraphics[width=0.16\textwidth]{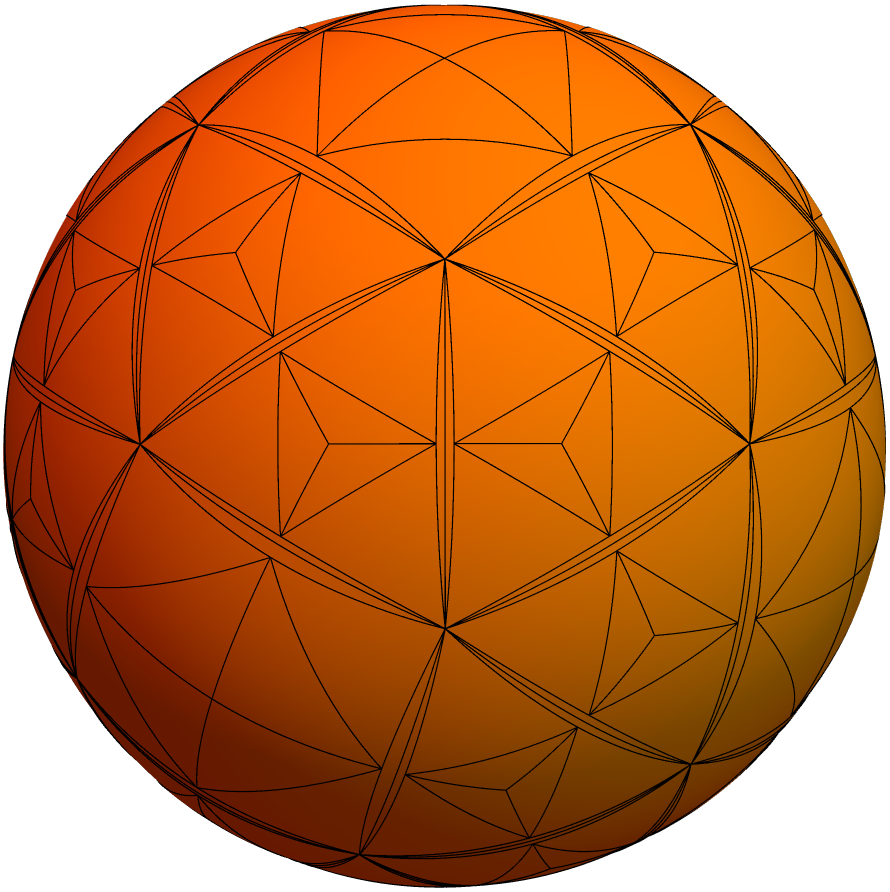}
\includegraphics[width=0.16\textwidth]{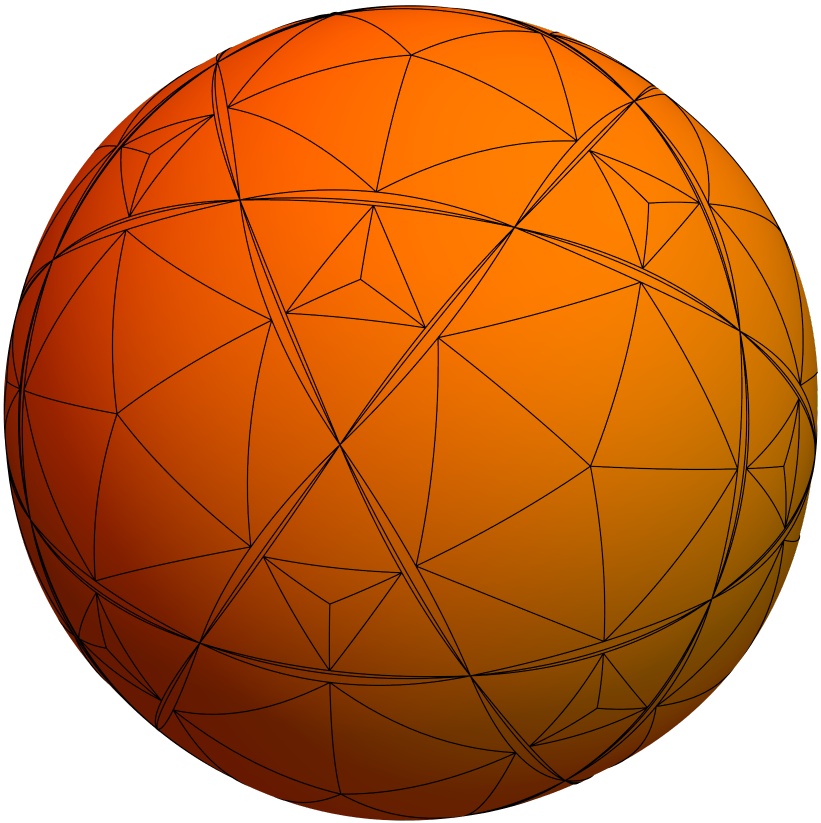}
\includegraphics[width=0.16\textwidth]{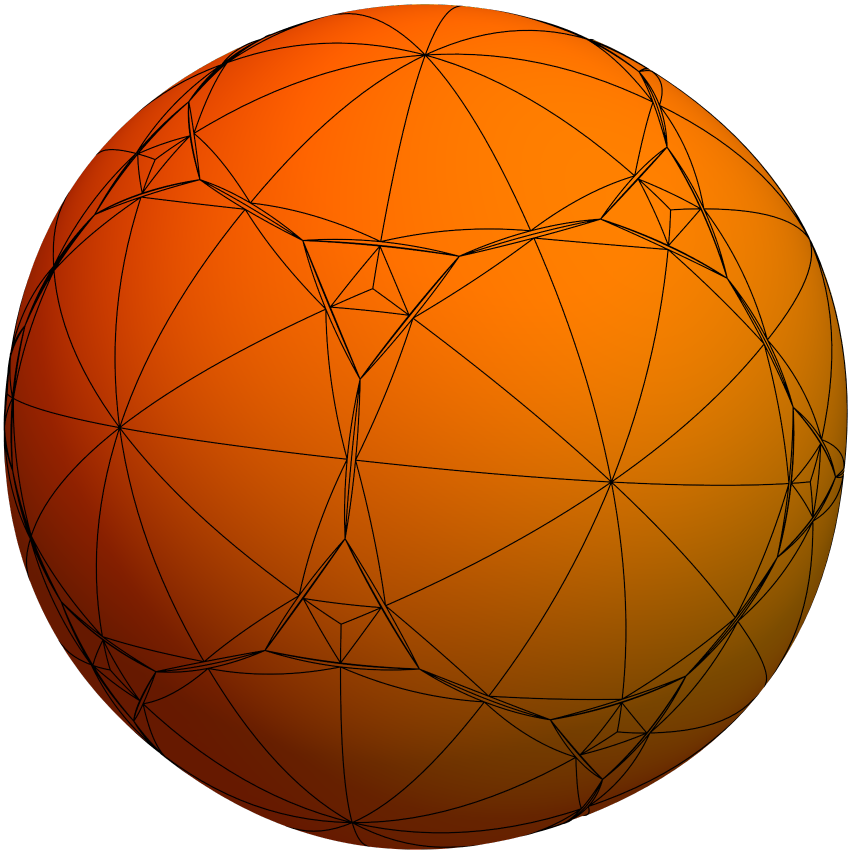}
\includegraphics[width=0.16\textwidth]{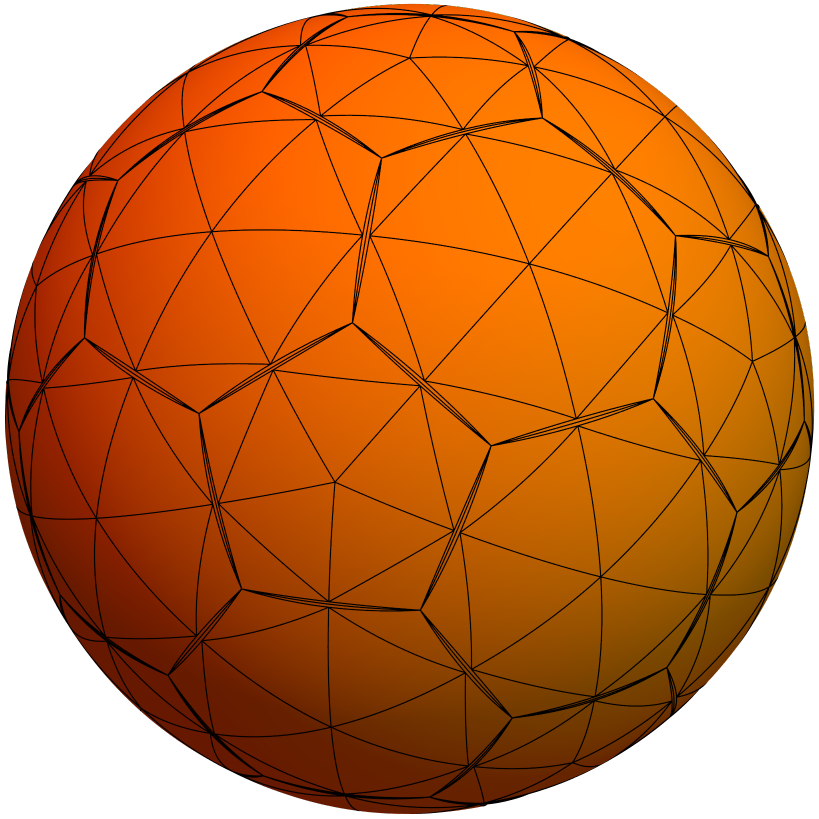}
\includegraphics[width=0.16\textwidth]{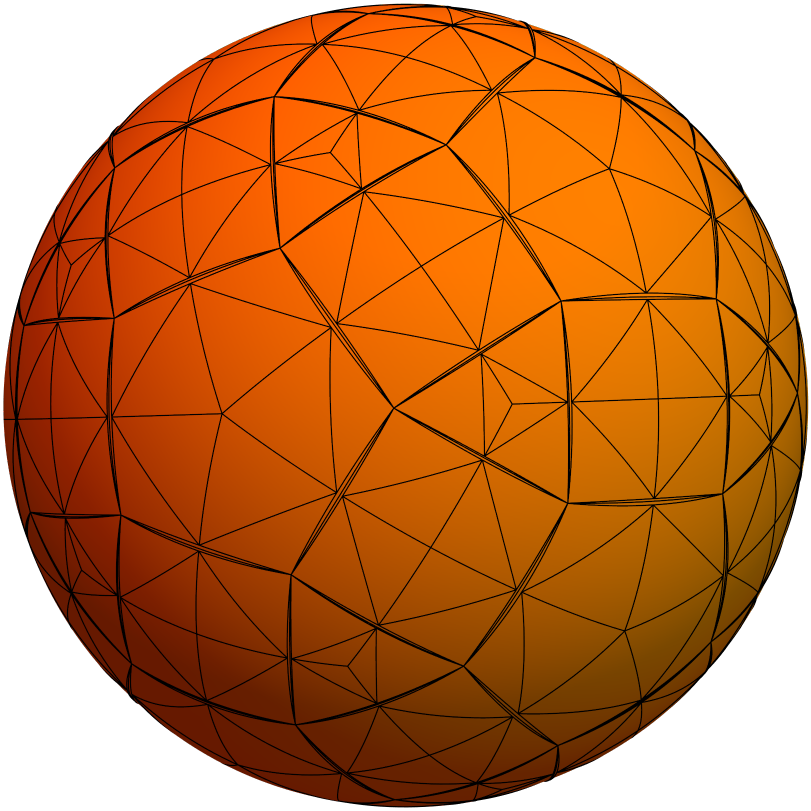}
\includegraphics[width=0.16\textwidth]{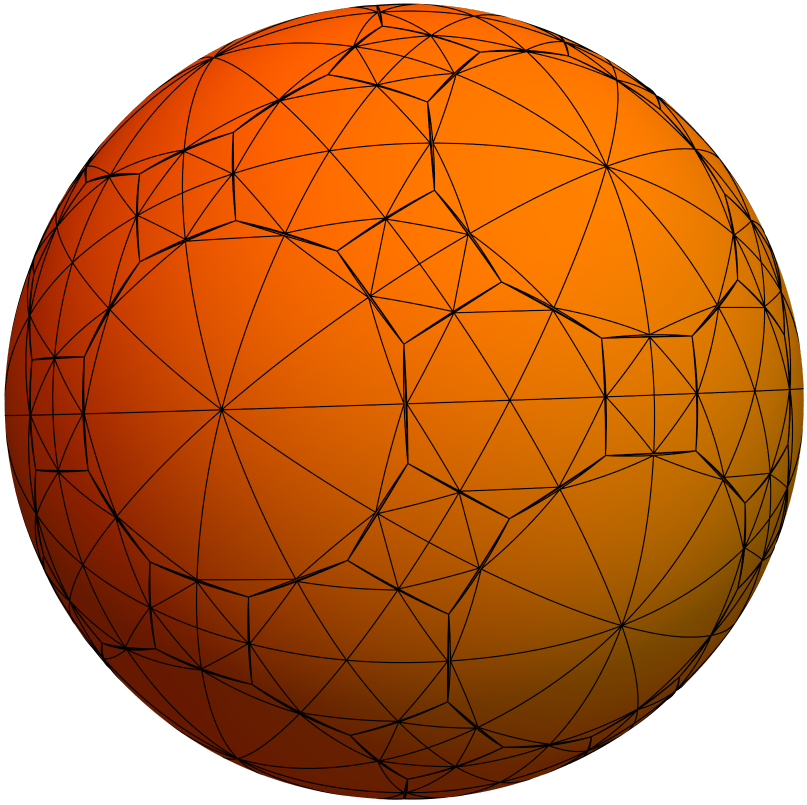}
\includegraphics[width=0.16\textwidth]{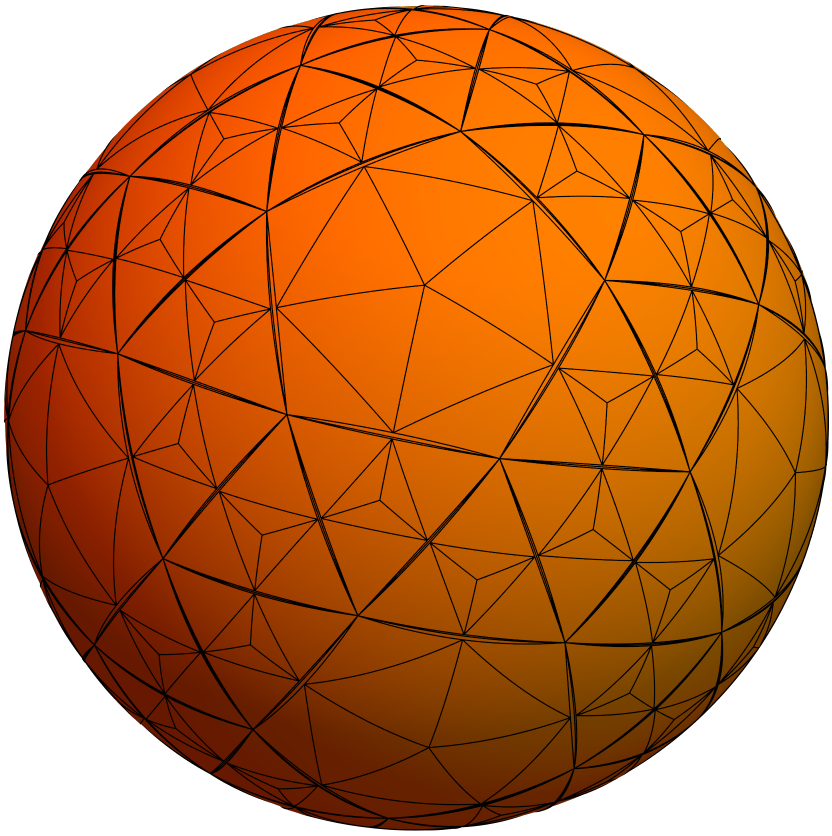}
\caption{\(G^1\) approximants corresponding to the five Platonic solids—the tetrahedron, cube, octahedron, dodecahedron, and icosahedron—and the thirteen Archimedean solids—the truncated tetrahedron, cuboctahedron, truncated cube, truncated octahedron, rhombicuboctahedron, truncated cuboctahedron, snub cube, icosidodecahedron, truncated dodecahedron, truncated icosahedron, rhombicosidodecahedron, truncated icosidodecahedron, and snub dodecahedron—for $x_1=c r$.}
\label{fig:UniformSolids}
\end{center}
\end{figure}

\begin{table}[h!]
    \centering
    \begin{tabular}{|l| c| c| |l| c| c|}
        \hline
        \textbf{Polyhedron} & $d_r$ &  $K \setminus \{0\}$ & \textbf{Polyhedron} & $d_r$ &  $K\setminus \{0\}$ \\
        \hline
        \hline
        Tetrahedron & $6{.}67 \cdot 10^{-1}$ & $ [0{.}11, 9{.}00]$ & 
        Rhombicuboctahedron & $1{.}09 \cdot 10^{-2}$ & $ [0{.}74, 1{.}34]$ \\
        Cube & $1{.}55\cdot 10^{-1}$ & $ [0{.}33, 3{.}00]$ & 
        Truncated cuboctahedron & $1{.}83 \cdot 10^{-2}$ & $ [0{.}68, 1{.}47]$ \\
        Octahedron & $1{.}55\cdot 10^{-1}$ & $[0{.}33, 3{.}00]$ & 
        Snub cube &$ 1{.}32\cdot 10^{-2}$ & $ [0{.}72, 1{.}38]$ \\
        Dodecahedron & $2{.}65\cdot 10^{-2}$ & $ [0{.}63, 1{.}58]$ & 
        Icosidodecahedron & $1{.}31\cdot 10^{-2}$ & $ [0{.}72, 1{.}38]$ \\
        Icosahedron & $2{.}65\cdot 10^{-2}$ & $ [0{.}63, 1{.}58 ]$& 
        Truncated dodecahedron & $1{.}56\cdot 10^{-2}$ & $ [0{.}70, 1{.}42]$ \\
        Truncated tetrahedron & $2{.}19\cdot 10^{-1}$ &$ [0{.}27, 3{.}67]$ & 
        Truncated icosahedron & $3{.}95\cdot 10^{-3}$ &$ [0{.}84, 1{.}19]$ \\
        Cuboctahedron & $2{.}06\cdot 10^{-2}$ & $ [0{.}50, 2{.}00]$ & 
        Rhombicosidodecahedron & $3{.}08\cdot 10^{-3}$ & $ [0{.}85, 1{.}17]$ \\
        Truncated cube & $7{.}61\cdot 10^{-2}$ & $ [0{.}46, 2{.}17]$ & 
        Truncated icosidodecahedron& $4{.}99\cdot 10^{-3}$ & $ [0{.}81, 1{.}22]$ \\
        Truncated octahedron & $3{.}28\cdot 10^{-2}$ & $ [0{.}60, 1{.}67]$ & 
        Snub dodecahedron & $3{.}58\cdot 10^{-3}$ & $ [0{.}84, 1{.}18]$ \\
        \hline
    \end{tabular}
    \caption{Limiting values of radial errors and Gaussian curvature ranges for the polyhedra as $x_1$ approaches $cr$. The surfaces are composed of patches constructed over regular polygons. In each patch, the segment corresponding to the triangle $\Delta_1$ has identically zero Gaussian curvature, while the table details the curvature ranges for the remaining segments of the patch.}
    \label{tab:errors}
\end{table}

There are also two infinite families of vertex-transitive polyhedra, namely the prisms and antiprisms, for which the same construction can be used as for the Archimedean solids. Although prisms and antiprisms with an \(n\)-gonal base have more faces as \(n\) increases, the resulting approximations become worse. This is due to the \(G^1\) continuity condition, which causes some patches to become increasingly elongated, while those corresponding to the lateral faces of the prism or antiprism become progressively smaller (see \Cref{fig:PrismSolids} and \Cref{fig:AntiPrismSolids}).
\begin{figure}[htb]
\begin{center}
\includegraphics[width=0.9\textwidth]{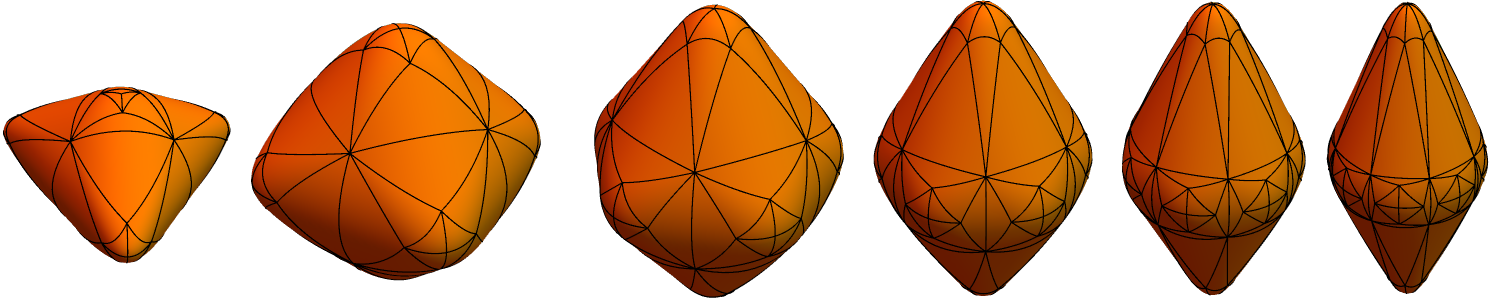}
\caption{\(G^1\) approximants induced by prisms, from the triangular prism to the octagonal prism.}
\label{fig:PrismSolids}
\end{center}
\end{figure}
\begin{figure}[htb]
\begin{center}
\includegraphics[width=0.9\textwidth]{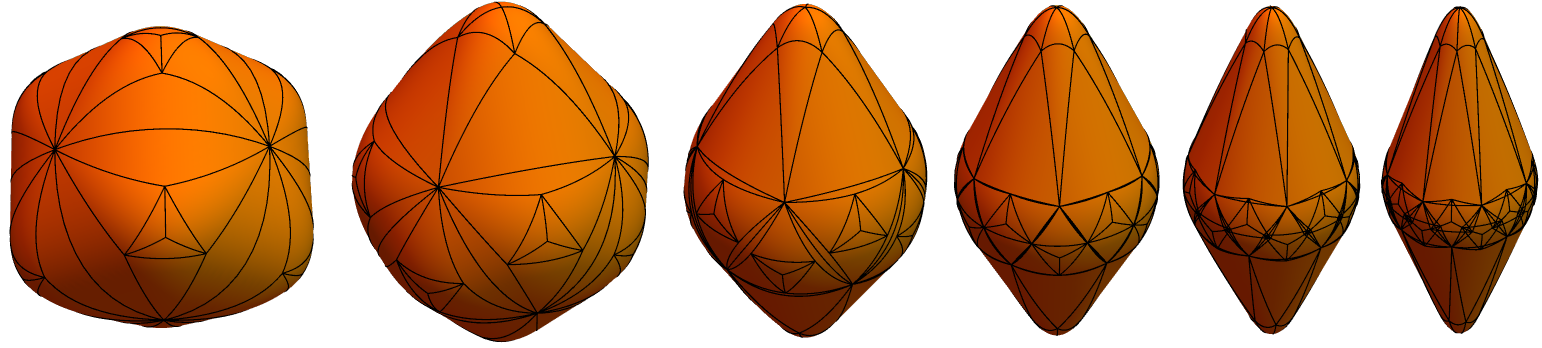}
\caption{\(G^1\) approximants induced by antiprisms, from the triangular antiprism to the octagonal antiprism.}
\label{fig:AntiPrismSolids}
\end{center}
\end{figure}

\section{Generalisation}\label{sec:Generalisation}

Using quadratic polynomial functions over triangular domains, we constructed a \(G^1\) spline that approximates the sphere. Although our construction is based on uniform polyhedra, the same approach can also be applied to certain Johnson solids.

The key requirements for the construction described in \Cref{sec:Triangulation} are that the faces of the polyhedron are regular polygons, that all vertices lie on a common sphere, and that the boundary curves constructed over the edges of the polyhedron lie in the planes determined by their endpoints and the coordinate origin. Furthermore, by construction, the normals to the tangent planes along the boundary curves also lie in the corresponding planes. The latter condition, together with \Cref{thm:characterisation}, ensures the \(G^1\) continuity of the spline.

Polyhedra satisfying these conditions are relatively rare. Moreover, if some faces have more vertices than others, the resulting spline tends to exhibit bumps above such faces, as observed in the case of prisms and antiprisms.

Among the Johnson solids, several candidates of interest include the pentagonal orthobirotunda, the elongated square gyrobicupola, and the metabigyrate rhombicosidodecahedron, which consist of triangles and pentagons, triangles and quadrilaterals, and triangles, quadrilaterals, and pentagons, respectively. The elongated square gyrobicupola is particularly close to the rhombicuboctahedron, differing only in that one cupola above the equatorial belt of squares is rotated by \( \tfrac{\pi}{4} \).

\section{Conclusion}\label{sec:Conclusion}

In surface modeling, it is essential to approximate shapes using objects that are as simple as possible while preserving a desired level of smoothness. Among such simple objects, quadratic polynomial triangular patches are of particular importance, and a central challenge is to join them in a way that achieves at least $G^1$ continuity.

In this paper, we presented a novel construction of a geometrically continuous quadratic spline over a suitable triangulation of an arbitrary uniform polyhedron. To the best of our knowledge, no such construction has previously appeared in the literature. The proposed approach yields a $G^1$ approximation of the sphere based entirely on quadratic triangular patches, thereby contributing a new construction to the theory of geometrically continuous splines.

Future work could focus on extending this construction to geometrically continuous quadratic splines over surfaces of higher genus. Another promising direction is to refine the construction so as to provide greater control over the curvature of the individual triangular patches, in particular to avoid the occurrence of parabolic patches.

\vskip3mm
{\noindent \sl Acknowledgments.}
The authors are grateful to Marjetka Knez and Emil \v{Z}agar for numerous fruitful discussions, valuable comments, and suggestions.
The first author was supported by the Slovenian Research and Innovation Agency program P1-0294. The second author was supported by the Slovenian Research and Innovation Agency program P1-0292.

\bibliographystyle{plain}
\bibliography{main}

\end{document}